\documentclass[journal,twoside,web]{ieeecolor}
\usepackage{generic}
\usepackage{cite}
\usepackage{amsmath,amssymb,amsfonts}
\usepackage{blkarray}
\usepackage{algorithmic}
\usepackage{graphicx}
\usepackage{algorithm,algorithmic}
\usepackage{hyperref}
\hypersetup{hidelinks=true}
\usepackage{textcomp}
\def\BibTeX{{\rm B\kern-.05em{\sc i\kern-.025em b}\kern-.08em
    T\kern-.1667em\lower.7ex\hbox{E}\kern-.125emX}}
\usepackage{epigraph}
\usepackage{subcaption} 

\usepackage{tikz}
\usetikzlibrary{decorations.pathmorphing}
\usetikzlibrary{positioning, shapes}
\tikzset{math3d/.style=
	{x= {(1cm,0cm)}, y={(0cm,1cm)},z={(-.5cm,-.7cm)}}}

\usepackage{tikz-cd}

\usepackage{amsthm}

\theoremstyle{plain}
\newtheorem{thm}{\textbf{Theorem}}

\newtheorem{prop}{\textbf{Proposition}}
\newtheorem{cor}{Corollary}

\newtheorem{defn}{\textbf{Definition}}
\newtheorem{assump}{\textbf{Assumption}}
\newtheorem{note}{\textbf{Note}}
\newtheorem{proper}{\textbf{Property}}

\newcommand{\scal}{\mathbf{\cdot}}

\begin{document}
\title{Existence of Solutions for Hybrid Dynamical Systems on Graph State-Space}
\author{Arthur Doliveira$^\star$, Christophe Roman$^\star$, Guillaume Graton$^{\star,\dagger}$ and  Mustapha Ouladsine$^\star$
\thanks{Manuscript received April ..., 2026... This project is supported by the French State, in the case of CASSIOPEE and ERMA project, as part of Investments for the Future Programme, now integrated into France 2030 with BPI France, and operated by ADEME.}
\thanks{$^{\star}$ A. Doliveira, C. Roman, G. Graton and M. Ouladsine are  with Aix-Marseille Université, Université de Toulon, CNRS, LIS (UMR 7020), Avenue Escadrille Normandie-Niemen, F-13397 Marseille Cedex 20, France
        ({\tt\small  arthur.doliveira@lis-lab.fr}, {\tt\small  christophe.roman@lis-lab.fr}, {\tt\small  mustapha.ouladsine@lis-lab.fr})}
\thanks{$^{\dagger}$ G. Graton is also with Centrale Méditerranée, Technopôle de Château-Gombert, 38 rue Frédéric Joliot-Curie, F-13451 Marseille Cedex 13, France
        ({\tt\small guillaume.graton@lis-lab.fr })}}

\maketitle

\begin{abstract}
This paper proposes a framework to ensure the existence of dynamical system trajectories in the state space of labeled, weighted, and attributed graphs. The evolution of such a system exhibits hybrid behavior: discrete jumps affecting the topology—the emergence and disappearance of vertices and edges over time—as well as vertex attributes and edge weights, combined with a continuous evolution of these same attributes and weights. To address the discrete behavior, an appropriate algebraic structure for the graph space is proposed; the analysis of its properties shows the existence of a new mathematical structure: a semi-vector space over the field of real numbers, whereas the literature only describes semi-vector spaces over semi-fields. The continuous behavior is modeled by differential equations. To facilitate formal treatment, the graph space is embedded, via a semi-linear mapping, into a new space, called variable-basis space, introduced in this work. The system's evolution model is then formulated within this space—the image of the graph space in the variable-basis space—under the framework of hybrid dynamical systems theory. A general result on the existence of solutions is established. Finally, this framework is applied to model and simulate the dynamics of the gut microbiota under antibiotic treatment followed by bacteriotherapy, where a generalized Lotka-Volterra (gLV) model describes the evolution of species abundances.
\end{abstract}


\section{Introduction}
\label{sec:introduction}

\noindent The term \emph{“system"} is nearly ubiquitous, referring to both physical and abstract entities. At its core lies the notion of interconnectedness between constituent elements—an idea already present in the definition provided by Pontus de Tyard in 1552. The mathematical definition of this term has succeeded from Bertalanffy \cite{von1968general}, even from Mesarovic, up to Backlund \cite{doliveira2024modeling}. The latter defines a system as a set of elements and a set of relations satisfying two conditions: the existence of at least two elements, and the existence of a path connecting any two elements of the system \cite{backlund2000definition}. To embody this definition, Torres et al. \cite{torres2021and} identify three primary mathematical structures: graphs, simplicial complexes, and hypergraphs, with graphs probably being the best known and most
widely used.

Whether natural or technological, most systems evolve over time. The study of these systems falls within the theory of dynamical systems, whose modern form originates with  Poincaré \cite{aubin2002writing}. This theory provides the concepts necessary for modeling, which at a minimum requires the definition of the state space, the time domain, and the evolution law \cite{brown2018modern}. An important class of dynamical systems is described by the initial value Cauchy problem, whose general form is:
\begin{equation}
	\label{ModeGen}
	\left\{ \begin{array}{l}
		\dot{\mathbf{x}} (t)=f(t,\mathbf{x}(t)), \quad \text{ with }\mathbf{x} \in \mathcal{H}, \text{ and } t\in \mathcal{T};\\
		\mathbf{x}(0)=\mathbf{x}_0\quad \text{ with } \quad \mathbf{x}_0\in\mathcal{H},
	\end{array}\right.
\end{equation}
where $\mathcal{H}$, $\mathcal{T}$, $\dot{\mathbf{x}}$, and $\mathbf{x}_0$ denote, respectively, the state space, the time domain, the state derivative induced by the law $f$, and the initial condition. The purpose of description~\eqref{ModeGen} is to capture the infinitesimal variations of the system. However, for it to have a well-defined mathematical meaning, the problem must be well-posed, and at a minimum, admit a solution. When $\mathcal{H} = \mathbb{R}^n$, Peano’s theorem~\cite{reid1975anatomy}, Carathéodory’s conditions~\cite{filippov2013differential}, as well as the results of Filippov and, more generally, Krasovskij~\cite{hajek1979discontinuous}, guarantee the existence of solutions. When $\mathcal{H}$ is a Banach space of infinite dimension, there also exist hypotheses ensuring the existence of solutions, such as Lipschitz-type conditions, at least locally in time~\cite{reid1975anatomy}. However, if the state space is that of graphs — a space whose elements are sets of vertices with attributes and edges with weights — then the following question arises:
\begin{center}
	\emph{``Does description~\eqref{ModeGen} retain a well-defined mathematical meaning in this space, i.e., is it well-posed and, at a minimum, does it admit solutions?   If not, what reformulation would guarantee its mathematical consistency?''}
\end{center}

This question arises, indeed, because the space of graphs presents a dual interest. First, from a theoretical perspective, the graph perfectly embodies the definition of Backlund's system \cite{backlund2000definition}, as previously mentioned (see also subsection 2.1 of \cite{doliveira2024modeling}). From a practical standpoint, this space is particularly well-suited for modeling dynamical systems in which both topology (components and links) and quantitative variables evolve simultaneously, as illustrated in Fig.~\ref{GrapDynam}. The need for such a space is implicitly expressed in a biological example provided by Hirsch \cite{hirsch1984dynamical}:

\epigraph{\emph{``... there are grave difficulties in modelling a developing embryo: the number of cells change, new tissues and organs appear whose configurations suddenly become an important part of the state. When such problems cannot be ignored, we usually have no recourse except simply to switch attention to a different system---a gas instead of a liquid, an adult organism instead of an embryo. (There is an interesting analogy to the mathematical operation of changing coordinates---except that here we do not know how to describe the coordinate changes!)''}}{\textemdash\ Morris H. Hirsch, 1984~\cite{hirsch1984dynamical}}%

Hirsch \cite{hirsch1984dynamical} provides additional examples from ecology, physics, and economics, illustrating that a system’s topology can evolve. He distinguishes between two approaches to modeling such phenomena: a change of perspective (“a gas instead of a liquid”) or a change of coordinates. The former entails no topological modification as the approach is macroscopic, whereas the latter focuses on detail; it is this second approach that is adopted in this work. Hirsch notes that, to our knowledge, no operation is known to allow such a change of coordinates—that is, a change in topology. In general, he emphasizes that such problems are intrinsically linked to the nature of the state space \cite{hirsch1984dynamical}.

  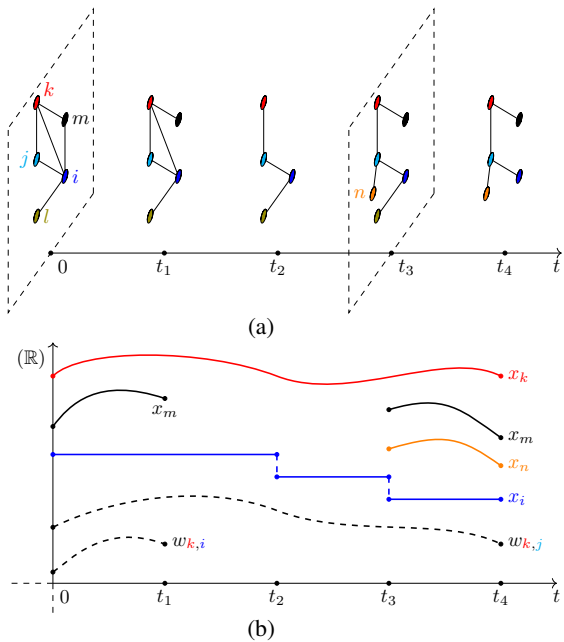
\begin{figure}[h]
	\begin{center}
		\begin{subfigure}[b]{.38\textwidth}
			\centering

  \resizebox{1.1\textwidth}{!}{
  
    \begin{tikzpicture}[math3d]
    \pgfmathsetmacro{\r}{300}
    
    \def \T {9.}
    \def \rB {4.5}
    \def \B {2}

    \draw[->] (0,0,1.5)--(\T,0,1.5) node[pos=.99,below]{$t$};
   \draw[dashed] (0,0,0)--(0,3.25,0) -- (0,3.25,3); 
    \draw[dashed] (0,0,0)--(0,0,3) -- (0,3.25,3) ;
    
    \filldraw[black] (0,0,1.5) circle (1pt);

    \draw[dashed] (6,0,0)--(6,3.25,0) -- (6,3.25,3); 
    \draw[dashed] (6,0,0)--(6,0,3) -- (6,3.25,3) ;
    
    \filldraw[black] (2,0,1.5) circle (1pt);
    \filldraw[black] (4,0,1.5) circle (1pt);
    \filldraw[black] (6,0,1.5) circle (1pt);
    \filldraw[black] (8,0,1.5) circle (1pt);
    
    \node at (2,0,1.5) [below] {$t_1$};
    \node at (4,0,1.5) [below] {$t_2$};
    \node at (6,0,1.5) [below right] {$t_3$};
    \node at (8,0,1.5) [below] {$t_4$};

\foreach \t in {0,2,4,6,8}{
\path coordinate (a1) at  (\t,1,1);
\path coordinate (a2) at  (\t,2,2);
\path coordinate (a3) at  (\t,3,2);
\path coordinate (a4) at  (\t,1,2);
\path coordinate (a5) at  (\t,2,1);
\path coordinate (a6) at  (\t,1.5,2.15);

\draw (a1)--(a2)--(a3);

\ifnum \t = 0
  \node at (a1) [right]{ $\textcolor{blue}{i}$};
  \node at (a2) [left]{ $\textcolor{cyan}{j}$};
  \node at (a3) [above right]{ $\textcolor{red}{k}$};
  \node at (a4) [right]{ $\textcolor{olive}{l}$};
  \node at (a5) [right]{ $\textcolor{black}{m}$};
\fi

\ifnum \t = 6
  \node at (a6) [left]{ $\textcolor{orange}{n}$};
\fi

\ifnum \t > 2
\else
  \draw (a1)--(a3);
  \draw (a5)--(a3);
  \draw[fill=black] plot[domain=0:2,samples=100]({\t}, {.1*sin(\r*\x)+2}, {.1*cos(\r*\x)+1});
\fi

\ifnum \t > 5
\draw (a5)--(a3);
  \draw[fill=black] plot[domain=0:2,samples=100]({\t}, {.1*sin(\r*\x)+2}, {.1*cos(\r*\x)+1});
\fi

\ifnum \t > 1
\else
  \draw (a1)--(a5);
\fi

\ifnum \t > 4
 \draw (a2)--(a6);
  \draw[fill=orange] plot[domain=0:2,samples=100]({\t}, {.1*sin(\r*\x)+1.5}, {.1*cos(\r*\x)+2.15});
\fi

\ifnum \t > 6
\else
  \draw (a1)--(a4);
  \draw[fill=olive] plot[domain=0:2,samples=100]({\t}, {.1*sin(\r*\x)+1}, {.1*cos(\r*\x)+2});
\fi

\draw[fill=blue] plot[domain=0:2,samples=100]({\t}, {.1*sin(\r*\x)+1}, {.1*cos(\r*\x)+1});
\draw[fill=cyan] plot[domain=0:2,samples=100]({\t}, {.1*sin(\r*\x)+2}, {.1*cos(\r*\x)+2});
\draw[fill=red] plot[domain=0:2,samples=100]({\t}, {.1*sin(\r*\x)+3}, {.1*cos(\r*\x)+2});



}

\node at (0,0,1.5) [below right] {$0$};


  \end{tikzpicture}

  }
			\vspace{-.6 cm}
			\caption{\label{TopoEvolve}}
		\end{subfigure}
		\\
		\begin{subfigure}[b]{.38\textwidth}
			\centering


  \resizebox{1.1\textwidth}{!}{
  \begin{tikzpicture}
      
      \def \T {9.}
      \def \X {4.3}
      \def \dx {.8}
      \def \dy {1.1}
      
      \draw[->] (0,0)--(\T,0) node[below]{$t$};
      
      \draw[dashed] (0,0)--(0,-.5*\dy);
      \draw[dashed] (0,0)--(-\dx,0);
      
      \draw[->] (0,0)--(0,\X) node[below left] {$(\mathbb{R})$};

    \filldraw[black] (2,0) circle (1pt);
     \filldraw[black] (4,0) circle (1pt);
    \filldraw[black] (6,0) circle (1pt);
    \filldraw[black] (8,0) circle (1pt); 
    \node at (0,0) [below right] {$0$};
    \node at (2,0) [below] {$t_1$};
    \node at (4,0) [below] {$t_2$};
    \node at (6,0) [below] {$t_3$};
    \node at (8,0) [below] {$t_4$};

    \draw[thick, red] 
            (0,3.7)         .. controls +(.5,.5)      and   +(-1.3,.5)
         .. (4,3.7)          .. controls +(1.3,-.5)    and   +(-1.,.5)
         .. (8, 3.7) node[right]{$x_k$};
     
     \filldraw[red] (8, 3.7) circle (1pt);
     \filldraw[red] (0,3.7) circle (1pt);
     
      \draw[thick, black]
           (0,2.8) .. controls +(.3,.4)      and   +(-1.2,.4)
        .. (2.,3.3)node[below]{$x_m$};
        
     \filldraw[black] (0,2.8) circle (1pt);
     \filldraw[black] (2.,3.3) circle (1pt);
     
     \draw[thick, black]
           (6,3.1) .. controls +(1.,.3)      and   +(-.7,.5)
        .. (8,2.6)  node[right]{$x_m$};
        
     \filldraw[black] (6,3.1) circle (1pt);
     \filldraw[black] (8,2.6) circle (1pt);
     
     \draw[thick, orange]
           (6,2.4) .. controls +(1.,.3)      and   +(-.7,.5)
        .. (8,2.1)  node[right]{$x_n$};
      \filldraw[orange] (6,2.4) circle (1pt);
     \filldraw[orange] (8,2.1) circle (1pt);
     
      \draw[thick, blue]   (0, 2.3) -- (4, 2.3);
      \draw[thick, blue, dashed]   (4, 2.3) -- (4,1.9);
      \draw[thick, blue]  (4,1.9) -- (6, 1.9);
      \draw[thick, blue, dashed]  (6, 1.9) -- (6, 1.5);
      \draw[thick, blue] (6, 1.5) -- (8, 1.5) node[right] {\textcolor{blue}{$x_i$}};  
      
      \filldraw[blue] (0, 2.3) circle (1pt);
      \filldraw[blue] (4, 2.3) circle (1pt);
      \filldraw[blue] (4,1.9) circle (1pt);
      \filldraw[blue] (6, 1.9) circle (1pt);
      \filldraw[blue] (6, 1.5) circle (1pt);
      \filldraw[blue] (8, 1.5) circle (1pt);
     
    \draw[thick, dashed] 
            (0,1.)         .. controls +(1,.5)      and   +(-1.3,.5)
         .. (4,1.3)          .. controls +(1.3,-.5)    and   +(-1.,.5)
         .. (8, .7) node[right]{$w_{\textcolor{red}{k},\textcolor{cyan}{j}}$};
         
     \filldraw (0,1.) circle (1pt);
     \filldraw (8, .7) circle (1pt);
     
      \draw[thick, dashed]
           (0,.2) .. controls +(.3,.2)      and   +(-1.1,.4)
        .. (2,.7) node[right]{$w_{\textcolor{red}{k},\textcolor{blue}{i}}$};
        
     \filldraw[black] (0,.2) circle (1pt);
     \filldraw[black] (2,.7) circle (1pt);
  \end{tikzpicture}
  }
			\vspace{-.6 cm}
			\caption{\label{AttributeEvolve}}
		\end{subfigure}
	\end{center}
	\vspace{-.2 cm}
	\caption{\label{GrapDynam}Illustration of the evolution of the state of a system represented by a graph in $\mathbb{G}$. Part (a) shows the evolution of the topology, while part (b) describes changes in the attributes of certain vertices and the weights of specific edges.}
\end{figure} 

The global behavior of such a system is illustrated in Fig.~\ref{GrapDynam}, where, in addition to the evolving topology (system components and the connections between them), the component attributes and the interaction weights can evolve simultaneously. By treating the topological evolution, as well as any discrete jumps in attributes and weights, as discrete events  and considering the continuous dynamics of these same attributes and weights, the system's evolution can be formulated within the framework of hybrid dynamical systems \cite{goebel2012hybrid}. To address the discrete behavior, we introduce a graph state space endowed with an appropriate algebraic structure. The analysis of this structure's properties reveals the existence, to the best of our knowledge, of a novel mathematical structure: a semi-vector space over the field of real numbers, thereby addressing a gap in the literature where only semi-vector spaces over semi-fields have been studied \cite{la2024semi,janyvska2007semi,la2021semi}. In parallel, the continuous dynamics of the attributes and weights are modeled by differential equations. To facilitate formal treatment, the graph space is embedded, via a semi-linear mapping, into a new space called the variable-basis space. In this representation, the evolution of attributes takes the form illustrated in Fig.~\ref{AttribDynam}; furthermore, this space provides the framework within which the existence of solutions is established.

To illustrate the relevance of this new framework, we apply it to the dynamics of the gut microbiota, where the evolution of species abundances is governed by the generalized Lotka-Volterra (gLV) model, specifically within the context of antibiotic treatment followed by bacteriotherapy \cite{jones2020navigation}. The formulation of the model within the graph state space, and more specifically its image in the variable-basis space, addresses a major challenge in microbiota modeling (see Section \ref{MicrobioteApp}). We now return to the central question of this paper by examining the relevant literature.

\subsection{\label{Litterature}Selected References}

\noindent The work of  Erd\H{o}s and  Rényi \cite{erdHos1960evolution} is, to our knowledge, the first study devoted to evolving graphs. They analyzed changes in graph properties as a function of structural parameters, such as the number of edges, rather than over time. From a temporal perspective, subsequent studies have addressed topological inference \cite{lebre2010statistical,coutino2020state}, identification of invariant properties \cite{zecevic2010control}, and learning on temporal graphs \cite{zambon2023graph,li2024state}. In this article, we focus on modeling the evolution of a system represented by a graph. The literature generally distinguishes four types of dynamic graph evolution \cite{harary1997dynamic}: (i) discrete topology evolution; (ii) discrete evolution of topology, attributes, and weights; (iii) continuous evolution of attributes and weights; and (iv) discrete topology evolution combined with continuous evolution of attributes and weights.

For the first form of evolution, Harary and Gupta \cite{harary1997dynamic} model a dynamic graph as a discrete sequence of static graphs governed by a Markov chain with transitions determined by a probability matrix. The second form is exemplified by Mazloum et al. \cite{mazloum2025implementation}, who extended classical algebraic operations \cite{Harary1969,knauer2019algebraic}, initially defined on topology, to account for discrete changes in vertex attributes and edge weights. While this model has been applied to vessel dynamics, the algebraic properties of these operations have not been studied \cite{mazloum2025implementation}, so closure of the space is not guaranteed. 

The third type of evolution, which concerns only changes in vertex attributes and/or edge weights, can be illustrated by works on collective dynamics, such as those by Baumann et al. \cite{baumann2020periodic} and Zecevic and \v{S}iljak \cite{zecevic2010control,vsiljak2008dynamic}. The evolution law of these systems can be expressed in the general form of the Cauchy problem \eqref{ModeGen}. Since the topology remains constant, the attributes and/or weights are grouped into vectors representing the system's state. This formulation allows the description of infinitesimal changes in this state. Under these conditions, there exist assumptions ensuring the existence of solutions, as discussed previously.

Finally, the fourth type of evolution is illustrated by Despréaux and Maignan \cite{despreaux2009dynamical}, who model an agent-based network. Although not explicitly stated, the model falls within the framework of hybrid dynamical systems. The state space is defined with continuous dynamics given by a future-state generator and discrete dynamics governed by local transformation rules, specified algorithmically without an explicit algebraic structure. While suitable for modeling such evolution, the approach has limitations. First, it is difficult to reuse when the continuous dynamics are described by a differential equation, as no existence result is guaranteed. Second, the lack of an algebraic structure for the transformation rules prevents properties such as closure of the state space. Finally, the model does not account for external perturbations. Although it alternates between continuous and discrete dynamics, it is not formalized within a rigorous hybrid systems framework.

\subsection{Contributions \& Organization}

\noindent This paper proposes a framework for ensuring the existence of trajectories of dynamical systems in the graph state space. It captures the simultaneous evolution of vertex attributes and edge weights through both infinitesimal variations and discrete jumps, as well as discrete topological changes (i.e., the emergence or disappearance of vertices, edges, or subgraphs). The main contributions are:

\begin{enumerate}
	\item The development of a rigorous algebraic structure for the graph space that captures the simultaneous discrete evolution of topology, attributes, and weights (Fig.~\ref{GrapDynam}).By associating the internal composition laws of Mazloum et al.~\cite{mazloum2025implementation} with a multiplicative external composition law, and by formally establishing the resulting algebraic properties, this work identifies—to the best of our knowledge—a new mathematical structure: semi-vector spaces over fields—a structure previously undocumented in the literature, which only considers semi-vector spaces over semifields~\cite{la2024semi,janyvska2007semi,la2021semi};
	
	\item The introduction of a variable-base state-space, into which the graph space is embedded via a semi-linear mapping. This construction is developed to ensure analytical tractability and to enable the derivation of formal mathematical results. Due to the isomorphic nature of the semi-linear mapping between the graph space and its image in the variable-base space, the image inherits the algebraic properties of the graph space, providing a consistent framework for the system's analysis;
	
	\item The formulation of the evolution model of the system within the variable-basis space, grounded in the theory of hybrid dynamical systems. Within this framework, a general result guaranteeing the existence of solutions is established;
	
	\item The application of the proposed framework to the simulation of gut microbiota dynamics, particularly to capture the complex evolution of species under antibiotic treatment and bacteriotherapy. This approach enables a reformulation of the generalized Lotka–Volterra (gLV) model, adapting it to contexts where the system structure is subject to topological changes, thereby overcoming the ecological challenges raised by Hirsch \cite{hirsch1984dynamical}.
\end{enumerate}

The remainder of the paper is organized as follows. Section \ref{Descript} presents the graph space and its algebraic structure. Section \ref{VBS_Frame} introduces the variable-base space, into which the graph space is embedded via a semi-linear mapping. Section \ref{HybDynModel} describes the model within the framework of hybrid dynamics. Section \ref{MicrobioteApp} illustrates the application of the model to gut microbiota dynamics. Finally, Section \ref{Conc} presents the conclusions.

\section{\label{Descript}Graph State-Space}
\noindent This section introduces the state space of graphs, equipped with an algebraic structure induced by a discrete evolution—illustrated in Fig. \ref{GrapDynam}—that acts simultaneously on the topology, the vertex attributes, and the edge weights. We subsequently define the constituent elements of this set.

\subsection{\label{ElemGraphTheory}Elements of the Set of Graphs}
\noindent Let \(\mathcal{I} \subseteq \mathbb{N}\) be a countable set of labels, referred to as the universe, whose elements are arbitrary. We denote by \( 2^{\mathcal{I}} \) the power set of \( \mathcal{I} \), defined such that for any \( I \in 2^{\mathcal{I}} \), \( I \subseteq \mathcal{I} \). Both the empty set \( \varnothing \) and the set \( \mathcal{I} \) itself are elements of \( 2^{\mathcal{I}} \). Let \( \mathbb{V} = \{(i, x_i)\mid i \in \mathcal{I}\} \) be a set of points indexed by \( \mathcal{I} \), where \( x_i \in \mathbb{R} \) is an attribute associated with the point labeled \( i \).

\begin{defn}
	\label{GraphDef} 
	A labeled, weighted graph $G$ with attributes, built on $\mathbb{V}$, consists of a pair $(V_G, E_G)$, where $V_G \subseteq \mathbb{V}$ and $E_G$ are the sets of vertices and edges of $G$, respectively, such that: 
	\[
	V_G = \{(i, x_i) \mid i \in \mathcal{I}_G, \; x_i \in \mathbb{R}\}, \]
	\[
	E_G = \{(p, q, w_{p,q}) \mid (p, q) \in \mathcal{L}_G, \; w_{p,q} \in \mathbb{R}\}.
	\]

	Here, $\mathcal{I}_G \subseteq \mathcal{I}$ and $\mathcal{L}_G \subseteq \mathcal{I}_G \times \mathcal{I}_G$ are the sets of labels for vertices and edges, respectively. $x_i$ represents the attribute of vertex $i$, and $w_{p,q}$ the weight of edge $(p,q)$.
\end{defn}

Hereafter, any mention of a graph refers to a labeled, weighted graph with attributes, unless stated otherwise. Let  \(\mathbb{G}\) denote the set of all loopless graphs on \(\mathbb{V}\). Extending our results to graphs with loops is straightforward and will be discussed when relevant. We now introduce several specific types of graphs constructed on \(\mathbb{V}\).
\begin{defn}
	\label{EmptyGraphDef}
	The empty graph, denoted \(\emptyset_{\mathbb{G}} = (\varnothing,\varnothing)\), is a graph in \(\mathbb{G}\) with no vertices and no edges.
\end{defn}

\begin{defn}
	\label{zeroGraph}
	A zero graph \( G^0 = (V_{G^0}, E_{G^0}) \in \mathbb{G}\) is a graph such that \( V_{G^0} = \{(i, 0) \mid i \in \mathcal{I}_{G^0} \} \) and \( E_{G^0} = \{(p, q, 0) \mid (p,q) \in \mathcal{L}_{G^0} \} \). 
\end{defn}

\begin{defn} 
	\label{totalGraph}
	A total graph \(G^\text{t} = (V_{G^\text{t}}, E_{G^\text{t}}) \in \mathbb{G}\) is a graph with \(V_{G^\text{t}} = \mathbb{V}\).
\end{defn}

\begin{defn}
	\label{completeGraph}
	A complete graph \(G^c\) in \(\mathbb{G}\) is a graph in which every pair of vertices \(p\) and \(q\) from the vertex set \(V_{G^c}\) is connected by an edge \((p,q)\) in the edge set \(E_{G^c}\).
\end{defn}

\begin{defn}%
	\label{loopsGraph}
	A graph \(G^l\) in \(\mathbb{G}\) has loops if it contains at least one edge connecting a vertex to itself.
\end{defn}

\begin{note}
	In this paper, we distinguish a zero edge weight from the actual existence of the edge.  
	Indeed, for a graph \(G = (V_G, E_G)\) with 
	\(E_G = \{(p, q, w_{p,q}) \mid (p, q) \in \mathcal{L}_G\}\), 
	contrary to Šiljak \cite{vsiljak2008dynamic}, a weight \(w_{p,q} = 0\) does not necessarily imply that the edge \((p,q)\) is absent from \(E_G\) for any \((p,q) \in \mathcal{L}_G\).
\end{note}

For any graph \(G \in \mathbb{G}\), there exists an adjacency matrix  \(\mathbf{A}_G = (a_{p,q})_{p,q \in \mathcal{I}_G}\)  and a weight matrix  \(\mathbf{W}_G = (w_{p,q})_{p,q \in \mathcal{I}_G}\),  where \(\mathcal{I}_G\) denotes the set of indices of the vertices of \(G\).   A value \(a_{p,q} = 1\) indicates the existence of the edge \((p,q)\),  whereas \(a_{p,q} = 0\) indicates its absence.   Furthermore, if the adjacency matrix \(\mathbf{A}_G\) is symmetric,  \(G\) is considered undirected; otherwise, \(G\) is directed.

\subsection{\label{AlgOpGraph}Algebraic Operations on the Set of Graphs}
\noindent The two internal composition laws used in \cite{mazloum2025implementation}, namely the overlapping union and the overlapping intersection, are considered, and an external composition law is introduced.

\subsubsection{\label{AddUnion}Overlapping (Additive) Union}
The additive union (overlapping union in \cite{mazloum2025implementation}) combines set-theoretic union on the topology with addition on vertex attributes and edge weights sharing the same labels (Fig.~\ref{AddUnionIllust}). Its formal definition follows.

\begin{defn}
	\label{AddUnionDef}
	The additive union is a mapping ~  \(\cup_+:\mathbb{G}\times\mathbb{G}\to\mathbb{G}\) that assigns to two graphs \(X = (V_X, E_X)\) and \(Y = (V_Y, E_Y)\) the graph \(X \cup_+ Y = (V_X \cup_+ V_Y, E_X \cup_+ E_Y)\), where:
	
	\vspace{-1. mm}
	
	\begin{itemize}
		\item From the vertex sets \(V_X = \{(i, x_i)\}_{i \in \mathcal{I}_X}\) and \(V_Y = \{(i, y_i)\}_{i \in \mathcal{I}_Y}\), the resulting vertex set is:
		\vspace{-1. mm}
		\[
		\begin{split}
			V_X \cup_+ V_Y := &\left\{ (i, x_i + y_i) \mid i \in \mathcal{I}_X \cap \mathcal{I}_Y \right\}\\
			&\cup  \left\{ (i, x_i) \mid i \in \mathcal{I}_X \setminus (\mathcal{I}_X\cap\mathcal{I}_Y) \right\}\\
			&\cup \left\{ (i, y_i) \mid i \in \mathcal{I}_Y \setminus (\mathcal{I}_Y\cap\mathcal{I}_X) \right\};
		\end{split}
		\]
		\vspace{-1. mm}
		\item From the edge sets \(E_X = \left\{ (k,l, w_{k,l}) \mid (k,l) \in \mathcal{L}_X \right\}\) and \(E_Y = \left\{ (k,l, v_{k,l}) \mid (k,l) \in \mathcal{L}_Y \right\}\), the resulting edge set is:
		\vspace{-1. mm}
		\[
		\begin{split}
			E_X \cup_+ E_Y := &\left\{ (k,l, w_{k,l} + v_{k,l}) \mid (k,l) \in \mathcal{L}_X \cap \mathcal{L}_Y \right\} \\ 
			&\cup \left\{ (k,l, w_{k,l}) \mid (k,l) \in \mathcal{L}_X \setminus (\mathcal{L}_X \cap \mathcal{L}_Y) \right\} \\
			&\cup \left\{ (k,l, v_{k,l}) \mid (k,l) \in \mathcal{L}_Y \setminus (\mathcal{L}_Y \cap \mathcal{L}_X) \right\}.
		\end{split}
		\]
	\end{itemize} 
\end{defn}

Note that the cardinalities of \(V_X\) and \(E_X\) satisfy  \(|V_X| = |\mathcal{I}_X|\) and \(|E_X| = |\mathcal{L}_X|\),  and similarly for \(Y\). Thus, $|V_X \cup_+ V_Y| = |\mathcal{I}_X| + |\mathcal{I}_Y| - |\mathcal{I}_X \cap \mathcal{I}_Y|$ and $|E_X \cup_+ E_Y| = |\mathcal{L}_X| + |\mathcal{L}_Y| - |\mathcal{L}_X \cap \mathcal{L}_Y|$.

\begin{figure}[h]
	\centering
	\begin{subfigure}[b]{.36\textwidth}
		\centering
		  \resizebox{.82\textwidth}{!}{
  
   \begin{tikzpicture}
    
      \def \simx {4. cm}
     \def \simy {1.3 cm }
     \def \dh { 1.3 cm }
     
     \def \dm {3. cm }
     
     \def \refx {0 cm }
     \def \refy {0 cm }
     
     \draw[rounded corners=10pt, thick]  (-1.2*\simx+\refx,\refy-1.6*\dh)  rectangle (.8*\simx+\refx,\refy+\simy+.7*\dh) ;
     
     \coordinate (center1) at (\refx,\refy);
     
     \small
     
     \node  (vi) at (-\simx+\refx,\refy)            [style={thick, circle,draw=black!100, inner sep=0pt, minimum size=8.5 mm}] {$i,x_i$};
     \node  (vj) at (-\simx+\dh+\refx,\refy-\simy)  [style={thick,circle,draw=cyan!100, inner sep=0pt, minimum size=8.5 mm}] {$j,x_j$};
     \node  (vk) at (-\simx+\dh+\refx,\refy+\simy)  [style={thick,circle,draw=red!100, inner sep=0pt, minimum size=8.5 mm}] {$k,x_k$};
     
     \draw[thick] (vk) -- (vi) -- (vj);
     
     \node (signe1+) at (-\simx+.5*\dm+.5*\dh+\refx, \refy) {\Large $\cup_+$};
     
     \node  (vjf) at (-\simx+\dm+\refx,\refy-\simy)       [style={thick,circle,draw=cyan!100, inner sep=0pt, minimum size=8.5 mm}] {$j,0$};
     \node  (vkf) at (-\simx+\dm+\refx,\refy+\simy)       [style={thick,circle,draw=red!100, inner sep=0pt, minimum size=8.5 mm}] {$k,0$};
     
     \draw[thick] (vkf) -- (vjf);
     
     \node (signe1=) at (center1) {\Large $=$};
     
     \node  (vis) at (.3*\simx+\refx,\refy)            [style={thick,circle,draw=black!100, inner sep=0pt, minimum size=8.5 mm}] {$i,x_i$};
     \node  (vjs) at (.3*\simx+\dh+\refx,\refy-\simy)  [style={thick,circle,draw=black!100, inner sep=0pt, minimum size=8.5 mm}] {$j,x_j$};
     \node  (vks) at (.3*\simx+\dh+\refx,\refy+\simy)  [style={thick,circle,draw=black!100, inner sep=0pt, minimum size=8.5 mm}] {$k,x_k$};
     
     \draw[thick] (vks) -- (vis) -- (vjs) --(vks);
     
  \end{tikzpicture}
  
  }
		\vspace{-.2 cm}
		\caption{\label{AddLink}}
		\vspace{.2 cm}
	\end{subfigure}
	\begin{subfigure}[b]{.36\textwidth}
		\centering
		  \resizebox{.82\textwidth}{!}{
  
   \begin{tikzpicture}
     \def \simx {4. cm}
     \def \simy {1.3 cm }
     \def \dh { 1. cm }
     
     \def \dm {3. cm }
     
     \def \refx {0 cm }
     \def \refy {0 cm }
     
     \draw[rounded corners=10pt, thick]  (-1.2*\simx+\refx,\refy-2.3*\dh)  rectangle (.95*\simx+\refx,\refy+\simy+.9*\dh) ;
     
     \coordinate (center1) at (\refx,\refy);
     
     \node  (v1) at (-\simx+\refx,\refy+\simy)  [style={thick, circle,draw=black!100, inner sep=0pt, minimum size=8.5 mm}] {$i,x_i$};
     \node  (v2) at (-\simx+\refx,\refy-\simy)  [style={thick, circle,draw=cyan!100, inner sep=0pt, minimum size=8.5 mm}] {$j,x_j$};
     
     \draw[thick] (v1) --(v2);

     \node (signe1+) at (-\simx+.5*\dm+\refx, \refy) {\Large $\cup_+$};
     
     \node  (v2f) at (-\simx+\dm+\refx,\refy-\simy)       [style={thick, circle,draw=cyan!100, inner sep=0pt, minimum size=8.5 mm}] {$j,0$};
     \node  (v3) at (-\simx+\dm+\refx,\refy+\simy)       [style={thick, circle,draw=black!100, inner sep=0pt, minimum size=8.5 mm}] {$k,x_k$};
     
     \draw[thick] (v3) -- (v2f);
     
     \node (signe1=) at (center1) {\Large $=$};
     
     \node  (v2s) at (\dh+\refx,\refy-\simy)  [style={thick, circle,draw=black!100, inner sep=0pt, minimum size=8.5 mm}] {$j,x_j$};
     \node  (v1s) at (\dh+\refx,\refy+\simy)  [style={thick, circle,draw=black!100, inner sep=0pt, minimum size=8.5 mm}] {$i,x_i$};
     
     \node  (v3s) at (\dh+.7*\dm+\refx,\refy+\simy)  [style={thick, circle,draw=black!100, inner sep=0pt, minimum size=8.5 mm}] {$k,x_k$};
     
     \draw[thick] (v1s) -- (v2s) -- (v3s);
  \end{tikzpicture}
  
  }
		\vspace{-.2 cm}
		\caption{\label{AddNode}}
		\vspace{.2 cm}
	\end{subfigure}
	\begin{subfigure}[b]{.36\textwidth}
		\centering
		  \resizebox{1.17\textwidth}{!}{
  
   \begin{tikzpicture}
     \def \simx {6.5 cm}
     \def \simy {1.5 cm }
     \def \dh { 1. cm }
     
     \def \dm {4. cm }
     
     \def \refx {0 cm }
     \def \refy {0 cm }
     
      \draw[rounded corners=10pt, thick]  (-\simx -\dh + \refx,\refy-\dh-\simy)  rectangle (.14*\simx+\refx+4*\dh,\refy +\dh +\simy) ;
     
     \coordinate (center1) at (\refx,\refy);
     
     \node  (v2) at (-\simx+\refx,\refy)            [style={thick, circle,draw=red!100, inner sep=0pt, minimum size=8.5 mm}] {$k,x_k$};
     \node  (v3) at (-\simx+\dh+\refx,\refy-\simy)  [style={thick, circle,draw=black!100, inner sep=0pt, minimum size=8.5 mm}] {$j,x_j$};
     \node  (v1) at (-\simx+\dh+\refx,\refy+\simy)  [style={thick, circle,draw=cyan!100, inner sep=0pt, minimum size=8.5 mm}] {$i,x_i$};
     
     \node  (e12) at (-\simx+ .5*\dh +\refx,\refy +.5*\dh) [above left]{$w_{i,k}$};  
     
     \draw[thick, blue] (v1) -- (v2);
     \draw[thick] (v2) -- (v3);
     
     \node (signe1+) at (-\simx+.5*\dm+.5*\dh+\refx, \refy) {\Large $\cup_+$};
     
     \node  (v2f) at (-\simx+\dm+\refx,\refy)          [style={thick, circle,draw=red!100, inner sep=0pt, minimum size=8.5 mm}] {$k,y_k$};
     \node  (v1f) at (-\simx+\dm+\refx,\refy+\simy)    [style={thick, circle,draw=cyan!100, inner sep=0pt, minimum size=8.5 mm}] {$i, y_i$};
     \node  (v4f) at (-\simx+\dm+1.5*\dh+\refx,\refy)  [style={thick, circle,draw=black!100, inner sep=0pt, minimum size=8.5 mm}] {$q,y_q$};
     
     \node  (v5f) at (-\simx+\dm+1.5*\dh+\refx,\refy+\simy)  [style={thick, circle,draw=black!100, inner sep=0pt, minimum size=8.5 mm}] {$p,y_p$};
     
     \node  (g12) at (-\simx+\dm+\refx,\refy+.5*\dh) [above left]{$\omega_{i,k}$}; 
     
     \draw[thick, blue] (v1f) -- (v2f);
     \draw[thick] (v2f) -- (v4f)--(v5f);
     
     \node (signe1=) at (center1) {\Large $=$};
     
     \node  (v2s) at (.14*\simx+\refx,\refy)            [style={thick, circle,draw=red!100, inner sep=0pt, minimum size=8.5 mm}] {$k,\textcolor{red}{z_k}$};
     \node  (v3s) at (.14*\simx+\dh+\refx,\refy-\simy)  [style={thick, circle,draw=black!100, inner sep=0pt, minimum size=8.5 mm}] {$j,x_j$};
     \node  (v1s) at (.14*\simx+\dh+\refx,\refy+\simy)  [style={thick, circle,draw=cyan!100, inner sep=0pt, minimum size=8.5 mm}] {$i,\textcolor{cyan}{z_i}$};
     \node  (v4s) at (.14*\simx+\refx+1.5*\dh,\refy)            [style={thick, circle,draw=black!100, inner sep=0pt, minimum size=8.5 mm}] {$q,y_q$};
     \node  (v5s) at (.14*\simx+\refx+3*\dh,\refy)            [style={thick, circle,draw=black!100, inner sep=0pt, minimum size=8.5 mm}] {$p,y_p$};
     
     \draw[thick, blue] (v1s) -- (v2s);
     \draw[thick, black]  (v2s) --(v3s);
     \draw[thick] (v2s)--(v4s)--(v5s);
     
     \node  (h) at (.14*\simx+\refx+.5*\dh,\refy+.5*\dh)  [above left]{$\textcolor{blue}{v_{i,k}}$};
     
     \large
     \node  (w) at (-\simx+\dm+\dh+\refx,\refy-.9*\dh){$\textcolor{cyan}{z_i} = x_i+y_i$}; 
     \node  (z) at (-\simx+\dm+\dh+\refx,\refy-1.4*\dh){$\textcolor{red}{z_k} = x_k+y_k$};
     \node  (lh) at (-\simx+\dm+1.15*\dh+\refx,\refy-2.*\dh){$\textcolor{blue}{v_{i,k}} = w_{i,k}+\omega_{i,k}$}; 
  \end{tikzpicture}
  
  }
  
  
		\vspace{-.2 cm}
		\caption{\label{AddSubGraph}}
	\end{subfigure}
	\caption{\label{AddUnionIllust}Illustration of additive union to add (a) an edge, (b) a vertex and (c) a subgraph.}
\end{figure}
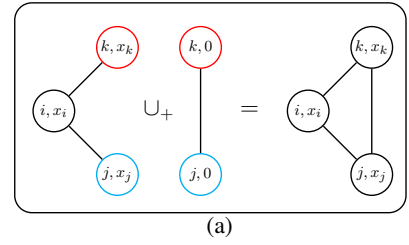
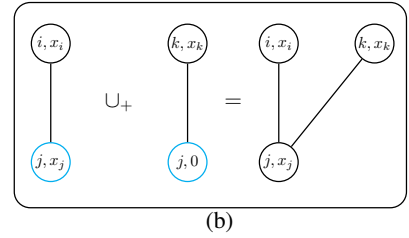
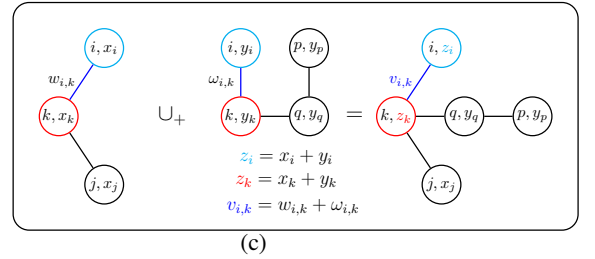 

\subsubsection{\label{AddInter}Overlapping (Additive) Intersection}
Like the additive union, the additive intersection (overlapping intersection in \cite{mazloum2025implementation}) combines set-theoretic intersection with addition (Fig.~\ref{AddInterIllust}). Its formal definition follows.

\begin{defn}
	\label{AddInterDef}
	The additive intersection is a map \(\cap_+: \mathbb{G} \times \mathbb{G} \to \mathbb{G}\) that assigns to two graphs \(X = (V_X, E_X)\) and \(Y = (V_Y, E_Y)\) the graph \(X \cap_+ Y = (V_X \cap_+ V_Y, E_X \cap_+ E_Y)\), where:
	\vspace{-1. mm}
	\begin{itemize}
		\item Given the vertex sets \(V_X = \{(i, x_i)\mid i \in \mathcal{I}_X\}\) and \(V_Y = \{(i, y_i)\mid i \in \mathcal{I}_Y\}\), the resulting vertex set is
		\vspace{-1.3 mm}
		\[
		V_X \cap_+ V_Y= \left\{(i, x_i + y_i) \mid i \in \mathcal{I}_X \cap \mathcal{I}_Y\right\}.
		\]
		\vspace{-3. mm}
		\item For  \(E_X = \left\{(k, l, w_{k,l}) \mid  (k, l) \in \mathcal{L}_X\right\}\) and \(E_Y = \left\{(k, l, v_{k,l}) \mid (k, l) \in \mathcal{L}_Y\right\}\), the resulting edge set is
		\vspace{-1.5 mm}
		\[
		E_X \cap_+ E_Y = \left\{(k, l, w_{k,l} + v_{k,l}) \mid (k, l) \in \mathcal{L}_X \cap \mathcal{L}_Y\right\}.
		\]
	\end{itemize}
\end{defn}

The cardinalities of the vertex and edge sets of \(X \cap_+ Y\) are, respectively: \(\lvert V_X \cap_+ V_Y \rvert = \lvert \mathcal{I}_X \cap \mathcal{I}_Y \rvert\) and \(\lvert E_X \cap_+ E_Y \rvert = \lvert \mathcal{L}_X \cap \mathcal{L}_Y \rvert\).

\begin{figure}[h]
  \centering
  \resizebox{.41\textwidth}{!}{
  
  \begin{tikzpicture}
     \def \simx {6.5 cm}
     \def \simy {1.5 cm }
     \def \dh { 1. cm }
     
     \def \dm {4. cm }
     
     \def \refx {0 cm }
     \def \refy {0 cm }
     
      \draw[rounded corners=10pt, thick]  (-\simx -\dh + \refx,\refy-\dh-\simy)  rectangle (.14*\simx+\refx+2*\dh,\refy +\dh +\simy) ;
     
     \coordinate (center1) at (\refx,\refy);
     
     \node  (v2) at (-\simx+\refx,\refy)            [style={thick, circle,draw=red!100, inner sep=0pt, minimum size=8.5 mm}] {$k,x_k$};
     \node  (v3) at (-\simx+\dh+\refx,\refy-\simy)  [style={thick, circle,draw=black!100, inner sep=0pt, minimum size=8.5 mm}] {$j,x_j$};
     \node  (v1) at (-\simx+\dh+\refx,\refy+\simy)  [style={thick, circle,draw=cyan!100, inner sep=0pt, minimum size=8.5 mm}] {$i,x_i$};
     
     \node  (e12) at (-\simx+ .5*\dh +\refx,\refy +.5*\dh) [above left]{$w_{i,k}$};  
     
     \draw[thick, blue] (v1) -- (v2);
     \draw[thick] (v2) -- (v3);
     
     \node (signe1+) at (-\simx+.5*\dm+.5*\dh+\refx, \refy) {\Large $\cap_+$};
     
     \node  (v2f) at (-\simx+\dm+\refx,\refy)          [style={thick, circle,draw=red!100, inner sep=0pt, minimum size=8.5 mm}] {$k,y_k$};
     \node  (v1f) at (-\simx+\dm+\refx,\refy+\simy)    [style={thick, circle,draw=cyan!100, inner sep=0pt, minimum size=8.5 mm}] {$i, y_i$};
     \node  (v4f) at (-\simx+\dm+1.5*\dh+\refx,\refy)  [style={thick, circle,draw=black!100, inner sep=0pt, minimum size=8.5 mm}] {$q,y_q$};
     
     \node  (v5f) at (-\simx+\dm+1.5*\dh+\refx,\refy+\simy)  [style={thick, circle,draw=black!100, inner sep=0pt, minimum size=8.5 mm}] {$p,y_p$};
     
     \node  (g12) at (-\simx+\dm+\refx,\refy+.5*\dh) [above left]{$\omega_{i,k}$}; 
     
     \draw[thick, blue] (v1f) -- (v2f);
     \draw[thick] (v2f) -- (v4f)--(v5f);
     
     \node (signe1=) at (center1) {\Large $=$};
     
     \node  (v3s) at (.14*\simx+\dh+\refx,\refy-\simy)  [style={thick, circle,draw=red!100, inner sep=0pt, minimum size=8.5 mm}] {$k,\textcolor{red}{z_k}$};
     \node  (v1s) at (.14*\simx+\dh+\refx,\refy+\simy)  [style={thick, circle,draw=cyan!100, inner sep=0pt, minimum size=8.5 mm}] {$i,\textcolor{cyan}{z_i}$};
     
    
    \draw[thick, blue] (v3s) -- (v1s);
     
     \node  (h) at (.14*\simx+\refx+.9*\dh,\refy-.2*\dh)  [above left]{$\textcolor{blue}{v_{i,k}}$};
     
     \large
     \node  (w) at (-\simx+\dm+\dh+\refx,\refy-.9*\dh){$\textcolor{cyan}{z_i} = x_i+y_i$}; 
     \node  (z) at (-\simx+\dm+\dh+\refx,\refy-1.4*\dh){$\textcolor{red}{z_k} = x_k+y_k$};
     \node  (lh) at (-\simx+\dm+1.15*\dh+\refx,\refy-2.*\dh){$\textcolor{blue}{v_{i,k}} = w_{i,k}+\omega_{i,k}$}; 
  \end{tikzpicture}
  
  }
  
  \caption{\label{AddInterIllust}Illustration of the additive intersection}
\end{figure}
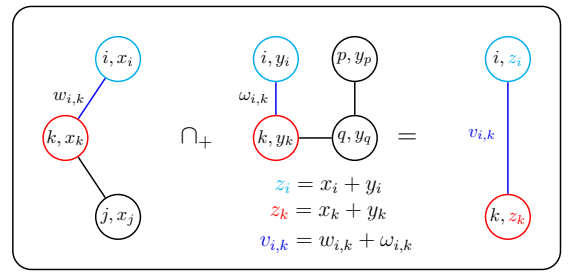

\subsubsection{\label{MultiExterLaw}External Multiplicative Composition Law}
Unlike the internal composition laws \(\cup_+\) and \(\cap_+\), the multiplicative external composition law applies only to attributes and weights. Its formal definition follows.
\begin{defn}
	\label{ActionLaw}
	Let \(\scal : \mathbb{R} \times \mathbb{G} \to \mathbb{G}\) be a multiplicative external composition law. 
	For any \(\alpha \in \mathbb{R}\) and any graph \(G = (V_G, E_G)\), it is defined by
	\[
	\alpha \scal G := (\alpha V_G, \alpha E_G), \text{ where}
	\]
	
	\[
	\alpha V_G \!=\! \{(i, \alpha x_i)\}_{i \in \mathcal{I}_G} \text{ and } \alpha E_G \!=\! \{(p, q, \alpha w_{p,q})\}_{(p, q) \in \mathcal{L}_G}.
	\]
	
	 For all \(\alpha, \beta \in \mathbb{R}\), \(G,H \in \mathbb{G}\), and for each \(\star \in \{\cup_+, \cap_+\}\), the law \(\scal\) satisfies:
	
	\begin{enumerate}
		\item Distributivity with respect to \(\star\):
		\begin{itemize}
			\item[--] $\alpha \scal (G \star H) = (\alpha \scal G) \star (\alpha \scal H)$,
			\item[--] $(\alpha + \beta) \scal G = (\alpha \scal G) \star (\beta \scal G)$;
		\end{itemize}
		\item Mixed Associativity: $(\alpha \beta) \scal G = \alpha \scal (\beta \scal G)$
		
		\item Multiplicative Identity: $1 \scal G = G$
	\end{enumerate}
\end{defn}

\subsection{\label{MainResult}Main Results}

\noindent  The main results, stated in theorems~\ref{GraphSetStructureU+} and~\ref{GraphSetStructuren+},  follow directly from the algebraic results established in subsections~\ref{DemTheo1}  and~\ref{DemTheo2}. Let \( \mathbb{R}_0^{+} := \{ x \in \mathbb{R} \mid x \ge 0 \} \).  Endowed with addition \( + \) and multiplication \( \scal \),  the structure \( (\mathbb{R}_0^{+}, +, \scal) \) is a semi-field  (see definitions~1 and~2 in~\cite{la2024semi}).  We now introduce the following definition.

\begin{defn}(Adapted from \cite{janyvska2007semi})
	\label{semi-vector-space}
	A semi-vector space over the semi-field $\mathbb{R}_0^+$ is a set $\mathcal{X}$ equipped with two operations: $\star: \mathcal{X} \times \mathcal{X} \mapsto \mathcal{X}$ and $\scal: \mathbb{R}_0^+ \times \mathcal{X} \mapsto \mathcal{X}$, which satisfy the following properties for all $\alpha, \beta \in \mathbb{R}_0^+$ and $x, y, z \in \mathcal{X}$:
	
	\begin{enumerate}
		\item  Associativity: $x \star (y \star z) = (x \star y) \star z$; 
		\item  Commutativity: $x \star y = y \star x$; 
		\item  Associativity of scalar multiplication: 
		\[
		(\alpha  \beta) \scal x = \alpha \scal (\beta \scal x);
		\]
		\item  Distributivity of scalar multiplication:
		 \[
		 \alpha \scal (x \star y) = \alpha \scal x \star \alpha \scal y;
		 \]
		\item  Distributivity of scalars: $(\alpha + \beta) \scal x = \alpha \scal x \star \beta \scal x$;
		\item  Identity element for scalar multiplication: $1 \scal x = x$.
	\end{enumerate}
	
	  In particular:
	\begin{itemize}
		\item[--]  A semi-vector space \(\mathcal{X}\) with an identity element \(e\) for \(\star\), meaning that \(x \star e = x\) for all \(x \in \mathcal{X}\), is said to be complete;
		\item[--] A complete semi-vector space \(\mathcal{X}\) with no invertible elements, is said to be simple;
		\item[--] A semi-vector space \(\mathcal{X}\) with an additive cancellation law, \textit{i.e.}, \( x \star y = z \star y \) implies \( x = z \), is said to be regular.
	\end{itemize}
\end{defn}

In order to fully apprehend the following theorem, one should keep in mind definitions \ref{EmptyGraphDef}, \ref{AddUnionDef}, \ref{ActionLaw} and \ref{semi-vector-space}.

\begin{thm}
	\label{GraphSetStructureU+}
	The triple \( (\mathbb{G}, \cup_+, \scal) \) forms a complete and simple semi-vector space  over the semi-field \( \mathbb{R}_0^+ \), whose identity element is the empty graph  \( \emptyset_{\mathbb{G}} \).
\end{thm}

\begin{note}
	\label{Semi-vector-Note}  
	In the literature, a semi-vector space is always defined over a semi-field  \cite{la2024semi,janyvska2007semi,la2021semi}. It is the semi-field that imparts its semi-vectorial structure to the vector space. However, a semi-vector space is more general than a traditional vector space, as it relaxes certain conditions, such as the existence of an identity element, the existence of inverses, and the property of additive cancellation \cite{janyvska2007semi}. In Theorem \ref{GraphSetStructureU+}, we established that \((\mathbb{G},\cup_+)\) forms a complete and simple semi-vector space over the semi-ring \(\mathbb{R}_0^+\). Since the multiplication induced by the semi-field \(\mathbb{R}_0^+\) applies to both attributes and weights, which have values in \(\mathbb{R}\) (see Definition \ref{ActionLaw}), the resulting value remains in \(\mathbb{R}\). Thus, whether we consider the semi-field \(\mathbb{R}_0^+\) or the field \(\mathbb{R}\), the properties of \((\mathbb{G},\cup_+)\) remain unchanged. In particular, it retains the closure property: for all \(G, H \in \mathbb{G}\) and \(\alpha \in \mathbb{R}\), \((\alpha \scal G) \cup_+ H \in \mathbb{G}\). Thus, \((\mathbb{G},\cup_+)\) can be regarded as a semi-vector space over both the semi-field \(\mathbb{R}_0^+\) and the field \(\mathbb{R}\).  
\end{note}
In view of Note~\ref{Semi-vector-Note}, Theorem~\ref{GraphSetStructureU+} can be extended to show that  \( (\mathbb{G}, \cup_+, \scal) \) is a complete and simple semi-vector space over the field  \( \mathbb{R} \). For the following theorem, we refer to  definitions~\ref{AddInterDef}, \ref{ActionLaw}, and~\ref{semi-vector-space},  Note~\ref{Semi-vector-Note}, as well as definitions~\ref{zeroGraph},  \ref{totalGraph}, and~\ref{completeGraph}, concerning the identity element.

\begin{thm}
	\label{GraphSetStructuren+}
	The triple \( (\mathbb{G}, \cap_+, \scal) \) forms a complete and simple semi-vector space over  the field \( \mathbb{R} \), whose identity element is the total, complete, and zero graph  \( \mathcal{G}_{t,c}^0 \).
\end{thm}

\begin{note}
	\label{IdWithLoops}
	If the set of graphs \( \mathbb{G} \) includes graphs with loops (cf. Definition~\ref{loopsGraph}),  the identity element becomes the total, complete, zero graph with loops, denoted  \( \mathcal{G}_{t,c}^{l,0} \), where \( l \) indicates the presence of loops.  By completeness, every vertex carries a loop;  by being zero, the weights of these loops are all zero.
\end{note}

The algebraic structure \((\mathbb{G}, \cup_+, \emptyset_{\mathbb{G}}, \cap_+, \mathcal{G}_{t,c}^0, \scal)\) over the field \(\mathbb{R}\) provides a precise framework for modeling the discrete changes occurring in the system, as illustrated in Fig.~\ref{GrapDynam}.  In the following, the graph space, denoted \(\mathbb{G}\), will refer to this algebraic structure over the field \(\mathbb{R}\). It is characterized by the fact that \((\mathbb{G}, \cup_+,\emptyset_{\mathbb{G}},\scal)\) and \((\mathbb{G}, \cap_+,  \mathcal{G}_{t,c}^0,\scal)\) are simple semi-vector spaces over  \(\mathbb{R}\). From a given state in \(\mathbb{G}\), it allows the determination of all other states of the system by using the properties of the internal composition laws \(\cup_+\) and \(\cap_+\) (see subsections \ref{DemTheo1} and \ref{DemTheo2}), as well as the multiplicative law \(\scal\). If \(\mathbb{G}\) includes graphs with loops, the structure becomes \((\mathbb{G}, \cup_+, \emptyset_{\mathbb{G}}, \cap_+, \mathcal{G}_{t,c}^{l,0},\scal)\) over the field \(\mathbb{R}\) (see Note \ref{IdWithLoops}). The next step is to establish the behavior of the additive union \(\cup_+\) in relation to the additive  intersection \(\cap_+\), or vice versa.

\begin{proper}
	\label{U+Etn+}  
	The algebraic structure \((\mathbb{G}, \cup_+, \cap_+)\) satisfies the following properties:  
	
	\vspace{-.5 mm}
	
	\begin{itemize}  
		\item Non-distributivity: For all \(X, Y, Z \in \mathbb{G}\),  
		\vspace{-.5 mm}
		\[
		\text{(1)} \hspace*{10 mm} X \cap_+ (Y \cup_+ Z) \neq (X \cap_+ Y) \cup_+ (X \cap_+ Z),
		\]
		\[
		\text{(2)} \hspace*{10 mm} X \cup_+ (Y \cap_+ Z) \neq (X \cup_+ Y) \cap_+ (X \cup_+ Z).
		\]  
		\vspace{-2.5mm}
		
		\item Order of operations: For all \(X, Y, Z \in \mathbb{G}\),  
		\vspace{-.5 mm}
		\[
		\text{(1)} \hspace*{10 mm} (X \cap_+ Y) \cup_+ Z \neq X \cap_+ (Y \cup_+ Z),
		\]
		\[
		\text{(2)} \hspace*{10 mm} (X \cup_+ Y) \cap_+ Z \neq X \cup_+ (Y \cap_+ Z).
		\]  
		\vspace{-2.5mm}
		
		\item Relation between \(\cup_+\) and \(\cap_+\): For all \(X, Y \in \mathbb{G}\),  
		\vspace{-.5 mm}
		\[
		X \cup_+ Y = (X \cap_+ Y) \cup (X\setminus (X\cap_+ Y)) \cup (Y\setminus (Y\cap_+ X)).
		\]  
	\end{itemize}
	The operation \(\setminus\) represents the graph difference operation. More precisely, let \(\mathcal{I}_X\), \(\mathcal{L}_X\), \(\mathcal{I}_Y\), and \(\mathcal{L}_Y\) denote the sets of indices for the vertices and edges of graphs \(X\) and \(Y\), respectively. The graph difference is then defined as  $X \setminus Y = (V^X_Y ,E^X_Y)$, where $V^X_Y = \{(i, x_i) \mid i \in \mathcal{I}_X \setminus (\mathcal{I}_X\cap \mathcal{I}_Y) \}$ and $E^X_Y = \{(k, l, w_{k,l}) \mid (k, l) \in \mathcal{L}_X \setminus (\mathcal{L}_X\cap\mathcal{L}_Y) \}.$ 
\end{proper}
Since \( \cup_+ \) and \( \cap_+ \) are not distributive with respect to each other (see Property \ref{U+Etn+}), the structure \( (\mathbb{G}, \cup_+, \cap_+) \) does not constitute a semi-algebra \cite{la2021semi}.

\subsection{\label{DemTheo1}Proof of Theorem \ref{GraphSetStructureU+}}
\noindent The approach consists in verifying whether the conditions of Definition~\ref{semi-vector-space} are satisfied. The conditions related to the multiplicative law $\scal$ hold by definition (cf. Definition~\ref{ActionLaw}), whereas those related to $\cup_+$ are stated as propositions, with proofs provided in Appendix~\ref{ProofPropTheoU+}.

\begin{prop}
	\label{PropAddUnion}
	The additive union \( \cup_+ \) on \( \mathbb{G} \) satisfies the following properties for all \( X, Y, Z \in \mathbb{G} \): 
	\begin{enumerate}
		\item  Commutativity: \( X \cup_+ Y = Y \cup_+ X \); 
		\item   Associativity:  \( (X \cup_+ Y) \cup_+ Z = X \cup_+ (Y \cup_+ Z) \),
		\item  Non-additive cancellation, \( X \cup_+ Y = Z \cup_+ Y \) does not imply \( X = Z \).
	\end{enumerate}
	 
\end{prop}

\begin{prop}
	\label{AddUnionIdentity}
	The identity element of the pair \((\mathbb{G}, \cup_+)\) is the empty graph \(\emptyset_{\mathbb{G}}\).
\end{prop}

\begin{prop}
	\label{Symetric+U}
	No element \( G \in \mathbb{G} \), except for \( \emptyset_{\mathbb{G}} \), which is its own inverse, has an inverse with respect to \( \cup_+ \).
\end{prop}

Since \((\mathbb{G}, \cup_+)\) is associative and commutative (Proposition~\ref{PropAddUnion}), admits the empty graph \(\emptyset_{\mathbb{G}}\) as identity (Proposition~\ref{AddUnionIdentity}), satisfies the axioms of scalar multiplication (Definition~\ref{ActionLaw}), and has no additive inverses with respect to \(\cup_+\) except for \(\emptyset_{\mathbb{G}}\) (Proposition~\ref{Symetric+U}), it follows that \((\mathbb{G}, \cup_+, \scal)\) is a complete and simple semi-vector space over \(\mathbb{R}_0^+\); moreover, since additive cancellation fails (Proposition~\ref{PropAddUnion}), \((\mathbb{G}, \cup_+,\scal)\) is not regular.

\subsection{\label{DemTheo2}Proof of Theorem \ref{GraphSetStructuren+}}
\noindent The approach is the same as that used in the proof of Theorem \ref{GraphSetStructureU+} in Subsection \ref{DemTheo1}. The propositions presented here are proved in Appendix \ref{ProoFTheo2}.

\begin{prop}
	\label{PropAddInter}
	The additive intersection \( \cap_+ \) on \( \mathbb{G} \) satisfies the following properties for all \( X, Y, Z \in \mathbb{G} \): 
	
	\begin{enumerate}
		\item Commutativity: \( X \cap_+ Y = Y \cap_+ X \);
		\item Associativity: \( (X \cap_+ Y) \cap_+ Z = X \cap_+ (Y \cap_+ Z) \); 
		\item  Non-additive cancellation: \( X \cap_+ Y = Z \cap_+ Y \) does not imply \( X = Z \).
	\end{enumerate}
\end{prop}

\begin{prop}
	\label{IdAddInter}
	The identity element of the pair \( (\mathbb{G}, \cap_+) \) is the total, complete, and zero graph\footnote{See Definitions~\ref{zeroGraph}, \ref{totalGraph}, and \ref{completeGraph} for the qualifiers of the identity element \( \mathcal{G}_{t,c}^0 \).} \( \mathcal{G}_{t,c}^0 \).
\end{prop}

\begin{prop}
	\label{SymInter+}
	Except for the total and complete graphs \( \mathcal{G}_{t,c} \), no graph in \( \mathbb{G} \) has an inverse under \( \cap_+ \).
\end{prop}

Since \((\mathbb{G}, \cap_+)\) is associative and commutative (Proposition~\ref{PropAddInter}), satisfies the axioms of scalar multiplication (Properties~\ref{ActionLaw}), admits the total, complete, and zero graph \(\mathcal{G}_{t,c}^0\) as identity (Proposition~\ref{IdAddInter}), and has no additive inverses for any graphs except for the graphs \(\mathcal{G}_{t,c}\) (Proposition~\ref{SymInter+}), it follows that \((\mathbb{G}, \cap_+,\scal)\) is a complete and simple semi-vector space over \(\mathbb{R}\); moreover, since additive cancellation fails (Proposition~\ref{PropAddInter}), \((\mathbb{G}, \cap_+,\scal)\) is not regular.

\section{\label{VBS_Frame}Variable-Basis State-Space}

\noindent B. Mazur points out in \cite{mazur2008one} that much of mathematics relies on representing the same object in different ways. Following this idea, we embed the graph space into a variable-basis space, thereby providing a flexible representation that facilitates the formulation of a hybrid dynamical model and the proof of existence of solutions (see \ref{ResExist}). We first describe the set of variable-basis elements.

\subsection{\label{VBS_set}Set of Variable-Basis Elements}

\noindent The variable-basis state-space introduced in this paper is inspired by the dimension-varying state-space of Cheng \cite{cheng2018linear,cheng2016equivalence}, but differs fundamentally: each vector in the variable-basis space is defined on an explicitly specified basis, whereas in the dimension-varying space, the basis is not explicitly given. Consequently, two vectors of the same dimension in the variable-basis space may lie on different bases, while in the dimension-varying space they share the same basis. 

An example of evolution in a variable-basis space is shown in Fig.~\ref{AttribDynam}, where the system state starts as a one-dimensional vector on basis \(b_1\), evolves on \((b_1, b_2)\), and ends on basis \(b_2\). Although the state vector’s dimension is the same at the beginning and end, the bases differ. Let us introduce the elements that constitute the variable-basis set.


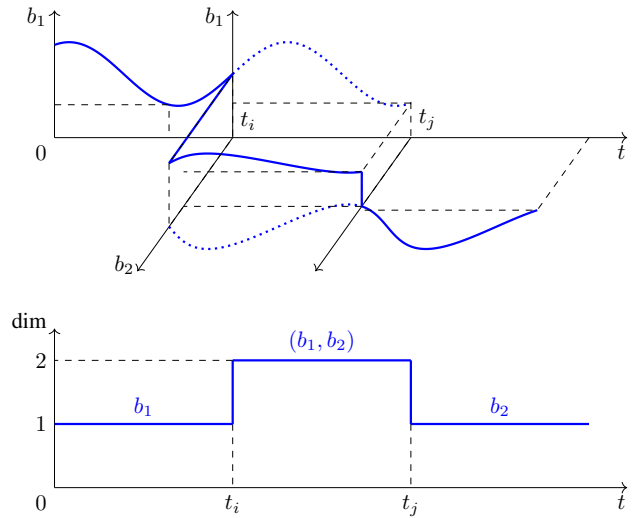
\begin{figure}[h]
  \centering
  \resizebox{.47\textwidth}{!}{
  
   \begin{tikzpicture}[math3d]
  
    \def \T {9.}
    \def \B {2}
    \def \rB {5.5}
    
    \def \F {2.8}
    
    \def  \db1{1}
    
    \pgfmathsetmacro{\r}{105}
    \draw[->] (0,0,0)--(\T,0,0) node[pos=.99,below]{$t$};
    \draw[->] (0,0,0)--(0,\B,0) node[pos=.95,left]{$b_1$};
    \draw[->] (\F,0,0) -- (\F,2,0) node[pos=.95,left]{$b_1$};
    \draw[->] (\F,0,0)--(\F,0,3) node[pos=.95,left]{$b_2$} ;
    \draw[->] (2*\F,0,0)--(2*\F,0,3) ; 
    
    \draw (0,0,0) node[pos=0.95, below left] {$0$};

    \draw[blue, line width=1pt] plot[domain=-\F:0,samples=500]( {\F+\x}, {.5*sin(\r*\x)+1}, 0); 
      
    \foreach \x in {-1}{
    \draw[dashed] ( 0, {.5*sin(\r*\x)+1}, 0) --( {\F+\x}, {.5*sin(\r*\x)+1}, 0)--( {\F+\x}, 0, 0);
    }
    
    \draw[blue, line width=1pt] plot[domain=0:\F,samples=500]( {\F+\x}, {.5*sin(\r*\x)+1}, {.5*sin(\r*\x)+2});
    \draw[blue, line width=1pt] (\F,{+1},0)--( {\F}, {1}, {2});
    \draw[dotted,blue, line width=1pt] plot[domain=0:\F,samples=500]( {\F+\x}, {.5*sin(\r*\x)+1}, 0);
    \draw[dotted,blue, line width=1pt] plot[domain=0:\F,samples=500]( {\F+\x}, 0, {.5*sin(\r*\x)+2});
    
    \draw(\F,0,0) node[above right] {$t_i$};
    
    \draw(2*\F,0,0) node[above right] {$t_j$};
    
\foreach \x in {0,\F}{
  \draw[dashed] ( {\F+\x}, {.5*sin(\r*\x)+1}, {.5*sin(\r*\x)+2})--( {\F+\x}, 0, {.5*sin(\r*\x)+2})--( {\F+\x}, 0, 0)--( {\F+\x}, {.5*sin(\r*\x)+1}, 0)--cycle;
\draw[dashed] ( {\F+\x}, 0, {.5*sin(\r*\x)+2}) -- ( {\F}, 0, {.5*sin(\r*\x)+2}) ;
\draw[dashed] ( {\F+\x}, {.5*sin(\r*\x)+1}, 0) -- ( {\F}, {.5*sin(\r*\x)+1}, 0) ;
\draw[dashed] ( {\F+\x}, {.5*sin(\r*\x)+1}, {.5*sin(\r*\x)+2})--( {\F}, {.5*sin(\r*\x)+1}, {.5*sin(\r*\x)+2});
}

   \draw[blue, line width=1pt] plot[domain=\F:2*\F,samples=500]( {\F+\x}, 0, {.5*sin(\r*\x)+2});
   \draw[blue,line width=1pt] ( {2*\F}, {.5*sin(\r*\F)+1}, {.5*sin(\r*\F)+2})--( {2*\F}, 0, {.5*sin(\r*\F)+2});
  \foreach \x in {2*\F}{
    \draw[dashed] ( {\F+\x}, 0, 0) --( {\F+\x}, 0, {.5*sin(\r*\x)+2})--(2*\F, 0, {.5*sin(\r*\x)+2});
   }
   
   
  
  \draw[->] (0,-\rB,0)--(\T,-\rB,0) node[pos=.99,below]{$t$};
  \draw[->] (0,-\rB,0)--(0,-\rB+\B+.5,0) node[pos=.95,above left]{dim};
  
  \node at (0,-\rB,0) [pos==.95, below left] {$0$};
  \draw[dashed] (\F,-\rB + 1,0) -- (\F,-\rB,0) node[below] {$t_i$};
  \draw[dashed] (2*\F,-\rB + 1,0) -- (2*\F,-\rB,0) node[below] {$t_j$};
  
  \draw[blue, line width=1pt] (0,-\rB + 1,0)--(\F,-\rB+1,0);
  \draw (.5*\F,-\rB + 1,0) node[above] {$\textcolor{blue}{b_1}$};
  
  \draw[blue, line width=1pt] (\F,-\rB + 2,0)--(2*\F,-\rB+2,0);
  \draw (1.5*\F,-\rB + 2,0) node[above] {$\textcolor{blue}{(b_1,b_2)}$};
  
  \draw[blue, line width=1pt] (2*\F,-\rB + 1,0)--(3*\F,-\rB+1,0);
  \draw (2.5*\F,-\rB + 1,0) node[above] {$\textcolor{blue}{b_2}$};
  
  \draw[blue, line width=1pt] (\F,-\rB+1,0)-- (\F,-\rB + 2,0);
  \draw[blue, line width=1pt]  (2*\F,-\rB + 2,0) -- (2*\F,-\rB+1,0);
  
  \draw[dashed] (\F,-\rB + 2,0) -- (0,-\rB + 2,0) node[left] {$2$};
  \draw (0,-\rB + 1,0) node[left] {$1$};

  \end{tikzpicture}

  }
  
  \vspace{-.2 cm}
  \caption{\label{AttribDynam}Illustration of the evolution of the state of a system in a variable-basis state space. The dashed curves represent the projections of the continuous curve onto the planes \((t, b_i)\), with \(i = 1, 2\).}
\end{figure}

Let $\mathcal{B}_v = \{\mathbf{b}_i \mid i \in \mathcal{I}\}$ be a set of linearly independent vectors, where $\mathcal{I} \subseteq \mathbb{N}$ is the universe introduced in Subsection \ref{ElemGraphTheory}. The power set of $\mathcal{B}_v$ is denoted by $2^{\mathcal{B}_v}$. For any subset $B \in 2^{\mathcal{B}_v}$, let $I(B) = \{i \in \mathcal{I} \mid \mathbf{b}_i \in B\}$ denote the set of indices associated with the vectors in $B$. The subset $B$ generates a subspace $\mathcal{V}(B)$, defined as:
\begin{equation}
	\label{attVector}
	\mathcal{V}(B) \!= \!\operatorname{span}(B) \!=\! \left\{ \sum_{i \in I(B)} \!\!\!\alpha_i \mathbf{b}_i \mid \alpha_i \in \mathbb{R}, \mathbf{b}_i \in B \right\}.
\end{equation}

Thus, for every $B \subseteq \mathcal{B}_v$, $\mathcal{V}(B)$ consists of all linear combinations of the vectors in $B$. With vector addition and scalar multiplication, it forms a vector space over $\mathbb{R}$. When $B = \varnothing$, $\mathcal{V}(B)$ is the trivial subspace containing only the empty vector $\emptyset_{\mathcal{V}}$, with null dimension, since its basis is $\varnothing$.
\begin{defn}
	\label{VBS_attributes}
	The variable-basis space generated by subsets of $\mathcal{B}_v$ is the space
	\vspace{-2 mm}
	\[
	\mathcal{V} \;:=\; \bigcup_{B \in 2^{\mathcal{B}_v}} \mathcal{V}(B).
	\]
	An element of $\mathcal{V}$, denoted $\mathbf{x}_B$, is said to be defined on the basis $B \in 2^{\mathcal{B}_v}$ and can be written as
	\[
	\mathbf{x}_B \;:=\; \sum_{k \in I(B)} x_k \mathbf{b}_k, \quad x_k \in \mathbb{R}.
	\]
\end{defn}

Let $\mathcal{B}_e = \{\mathbf{b}_k \otimes \mathbf{b}_l \mid \mathbf{b}_k, \mathbf{b}_l \in \mathcal{B}_v, \forall k,l \in \mathcal{I}\}$ denote the set generated by the tensor product $\mathcal{B}_v \otimes \mathcal{B}_v$, whose elements are referred to interchangeably as vectors or tensors \cite{burdick1995introduction}; $\mathcal{B}_e$ is linearly independent, and for any $D \in 2^{\mathcal{B}_e}$, let $I(D) = \{(k,l) \in \mathcal{I}^2 \mid \mathbf{b}_k \otimes \mathbf{b}_l \in D\}$ denote the index pairs associated with the vectors in $D$, and let the subset $D$ generate the subspace $\mathcal{W}(D)$ defined as:
\begin{equation}
	\label{weightVector}
	\begin{split}
		&\mathcal{W}(D) = \operatorname{span}(D) =\\
		&\left\{ \sum_{(k,l) \in I(D)} \beta_{k,l} \mathbf{b}_k \otimes \mathbf{b}_l \mid \beta_{k,l} \in \mathbb{R}, \mathbf{b}_k \otimes \mathbf{b}_l \in D \right\}.
	\end{split}
\end{equation}

For any $D \in 2^{\mathcal{B}_e}$, $\mathcal{W}(D)$, with vector addition and scalar multiplication, forms a vector space over $\mathbb{R}$. For each $D \in 2^{\mathcal{B}_e}$, we define the corresponding boolean-valued spaces as:
\begin{equation}
	\label{boolVect}
	\begin{split}
		&\mathcal{A}(D)= \operatorname{span}(D) = \\
		&\left\{ \!\!\sum_{(k,l) \in I(D)} \!\!\!\!\!\!\gamma_{k,l} \mathbf{b}_k \otimes \mathbf{b}_l \mid \gamma_{k,l} \in \{0,1\}, \mathbf{b}_k \otimes \mathbf{b}_l \in D \right\}.
	\end{split}
\end{equation}
When $D = \emptyset$, both $\mathcal{W}(D)$ and $\mathcal{A}(D)$ contain only the empty vectors $\emptyset_{\mathcal{W}}$ and $\emptyset_{\mathcal{A}}$, respectively, each of null dimension.

\begin{defn}
	\label{free-dim-wei}
	The variable-basis  space with real-valued coordinates, generated by subsets of $\mathcal{B}_e$, is the space
	\[
	\mathcal{W} \;:=\; \bigcup_{D \in 2^{\mathcal{B}_e}} \mathcal{W}(D).
	\]
	An element of $\mathcal{W}$, denoted $\mathbf{w}_D$, is said to be defined on the basis $D \in 2^{\mathcal{B}_e}$ and can be expressed as
	\[
	\mathbf{w}_D \;:=\; \sum_{(p,q) \in I(D)} w_{p,q} \, \mathbf{b}_p \otimes \mathbf{b}_q, \quad w_{p,q} \in \mathbb{R}.
	\]
\end{defn}

\begin{defn}
	\label{free-dim-struct}
	The variable-basis boolean-valued space generated by subsets of $\mathcal{B}_e$ is the space
	\[
	\mathcal{A} \;:=\; \bigcup_{D \in 2^{\mathcal{B}_e}} \mathcal{A}(D).
	\]
	An element of $\mathcal{A}$, denoted $\mathbf{a}_D$, is said to be defined on the basis $D \subset \mathcal{B}_e$ and can be expressed as
	\[
	\mathbf{a}_D \;:=\; \sum_{(p,q) \in I(D)} a_{p,q} \, \mathbf{b}_p \otimes \mathbf{b}_q, \quad a_{p,q} \in \{0,1\}.
	\]
\end{defn}
\begin{defn}
	\label{variable-basis-product}
	The variable-basis space formed by the Cartesian product of $\mathcal{V}$, $\mathcal{W}$, and $\mathcal{A}$ is the space
	\[
	\mathcal{D} \;:=\; (\mathcal{V}, \mathcal{W}, \mathcal{A}).
	\]
	An element of $\mathcal{D}$, denoted $X_B^{D,C}$, is defined on the bases $B \subset \mathcal{B}_v$, $D \subset \mathcal{B}_e$, and $C \subset \mathcal{B}_e$, and can be expressed as
	\[
	X_B^{D,C} \;:=\; (\mathbf{x}_B, \mathbf{w}_D, \mathbf{a}_C).
	\]
\end{defn}

The variable-basis state sets $\mathcal{V}$ and $\mathcal{W}$, equipped with vector addition and scalar multiplication, are not vector spaces over $\mathbb{R}$, since vector addition loses its meaning in these spaces. When the basis is not explicitly given, $\mathcal{V}$ becomes a dimension-varying state set \cite{cheng2018linear}. By introducing V-addition and V-subtraction \cite{cheng2020equivalence}, it acquires the structure of a pseudo-vector space (Proposition 2.3 in \cite{cheng2020equivalence}). In this paper, the operations defined in Subsection \ref{RealValueOp} endow $\mathcal{V}$ and $\mathcal{W}$ with the structure of a semi-vector space over $\mathbb{R}$ (Subsection \ref{AlgbStruc}).

\subsection{\label{FromSettoSet}Mapping from Graph Set $\mathbb{G}$ to Variable-Basis Set $\mathcal{D}$}
\noindent To transition from the graph-state set \(\mathbb{G}\) to the variable-basis state set \(\mathcal{D}\), we define a mapping \(\varphi : \mathbb{G} \to \mathcal{D}\) that associates each element of \(\mathbb{G}\) with an element of \(\mathcal{D}\) as:

\begin{equation}
	\label{mapGtoD}
	\begin{array}{rcl}
		\varphi : \mathbb{G} &\longrightarrow& \mathcal{D} \\
		G &\longmapsto& X_B = (\mathbf{x}_{B}, \mathbf{w}_{D}, \mathbf{a}_{D}),
	\end{array}
\end{equation}
where \(G = (V_G, E_G)\), $B \in 2^{\mathcal{B}_v}$ and $D = B \otimes B$. In detail, the components of $X_B$ under the mapping $\varphi$ are:

\emph{i)}- From the vertex set \(V_G = \{(i, x_i)\}_{i \in \mathcal{I}_G}\), we define the vector:
\begin{equation}
	\label{att-cross}
	\mathbf{x}_{B} = \sum_{i \in I(B)} x_i \mathbf{b}_i,
\end{equation}
where $B \subset \mathcal{B}_v$, $I(B) = \{i \in \mathcal{I} \mid \mathbf{b}_i \in B\}$, and $\mathcal{I}_G = I(B)$, with $\mathcal{I}$ referring to the universe.

\emph{ii)}- From the edge set $E_G = \{(p, q, w_{p,q})\}_{(p,q) \in \mathcal{L}_G}$, we extract the weight matrix and the adjacency matrix, both viewed as second-order tensors:
\begin{equation}
	\label{weiCross}
	\mathbf{w}_{D} = \!\!\!\sum_{(p,q) \in I(D)}\!\!\!\! w_{p,q} \mathbf{b}_p \otimes \mathbf{b}_q, \hspace{ 1mm} \mathbf{a}_{D} = \!\!\!\sum_{(p,q) \in I(D)}\!\!\!\! a_{p,q} \mathbf{b}_p \otimes \mathbf{b}_q,
\end{equation}
where \(I(D) = \mathcal{I}_G \times \mathcal{I}_G\). Let us recall, by definition, that \(I(D) = \{(p,q) \in \mathcal{I}^2 \mid \mathbf{b}_p \otimes \mathbf{b}_q \in D\}\).

\begin{defn}
	Let $\varphi \colon \mathbb{G} \to \mathcal{D}$ be the mapping defined in \eqref{mapGtoD}. Its direct image is the set $\mathcal{D}_{\mathbb{G}} = \varphi(\mathbb{G} )$, expressed as:
	\[
	\mathcal{D}_{\mathbb{G}} \! = \!\left\{ (\mathbf{x}_{B}, \mathbf{w}_{D}, \mathbf{a}_{D}) \!\in\! \mathcal{D} 
	\middle| \exists G \!\in\! \mathbb{G},\ \varphi(G) \!=\! (\mathbf{x}_{B}, \mathbf{w}_{D}, \mathbf{a}_{D}) \right\}.
	\]
	Equivalently,
	\[
	\mathcal{D}_{\mathbb{G}} = \left\{ (\mathbf{x}_{B}, \mathbf{w}_{D}, \mathbf{a}_{D}) \in \mathcal{D} 
	\middle| \forall B \subseteq \mathcal{B}_v,~ D = B \otimes B \right\},
	\]
	where $\mathbf{x}_{B},~ \mathbf{w}_{D},~ \mathbf{a}_{D}$ are defined in \eqref{att-cross} and \eqref{weiCross}, with $x_i, w_{p,q} \in \mathbb{R}$, $a_{p,q} \in \{0,1\}$, $i \in I(B)$, and $(p,q) \in I(D)$.
\end{defn}

\begin{prop}
	\label{Bijectionphi}
	The mapping \(\varphi\) defined in \eqref{mapGtoD} is a bijection from \(\mathbb{G}\) to \(\mathcal{D}_{\mathbb{G}}\).
\end{prop}

\begin{proof}
	It is sufficient to show that \( \varphi \) is both surjective and injective. By definition, since \( \mathcal{D}_{\mathbb{G}} = \varphi(\mathbb{G}) \), \( \varphi \) is surjective \cite{bourbaki2004theory}. Injectivity is simply established by identification.
\end{proof}

Since \(\varphi\) is bijective, its inverse, denoted \(\varphi^{-1}\), exists and is also bijective \cite{bourbaki2004theory}. It is defined as a mapping \(\varphi^{-1}:~\mathcal{D}_{\mathbb{G}} \to\mathbb{G}\). The compositions \(\varphi \circ \varphi^{-1}\) and \(\varphi^{-1} \circ \varphi\) yield the identity map. Then, the subspace $\mathcal{D}_{\mathbb{G}}$ is the space of graphs $\mathbb{G}$ embedded in the variable-basis space $\mathcal{D}$.

 \begin{defn}
	Let \(\emptyset_{\mathcal{D}} = (\emptyset_{\mathcal{V}}, \emptyset_{\mathcal{W}}, \emptyset_{\mathcal{A}})\) be the zero-dimensional element of the set \(\mathcal{D}\), where \(\emptyset_{\mathcal{V}}\) is the zero-dimensional vector in \(\mathcal{V}\), and \(\emptyset_{\mathcal{W}}\) and \(\emptyset_{\mathcal{A}}\) are the zero-dimensional tensors in \(\mathcal{W}\) and \(\mathcal{A}\), respectively.
\end{defn}

\begin{defn}
	\label{TotCompNull}
	Let \(X_{B} = (\mathbf{x}_{B}, \mathbf{w}_{D}, \mathbf{a}_{D})\) be an element of \(\mathcal{D}_{\mathbb{G}}\) with \(B \subset \mathcal{B}_v\) and \(D = B \otimes B\). The element \(X_B\) is said to be:
	
	\begin{enumerate}
		\item total, if and only if \(B = \mathcal{B}_v\); 
		\item complete, denoted by \(X_{B}^c = (\mathbf{x}_{B}, \mathbf{w}_{D}, \mathbf{a}_{D}^c)\), if, for all \((p, q) \in I(D)\), \(a_{p, q} = 1\) and \(a_{p, p} = 0\); 
		\item  null, denoted by \(X_{B}^0 = (\mathbf{x}_{B}^0, \mathbf{w}_{D}^0, \mathbf{a}_{D})\), if for all \(i \in I(B)\), \(x_i = 0\), and for all \((p, q) \in I(D)\), \(w_{p, q} = 0\);
		\item with loop(s), denoted by \(X_{B}^l = (\mathbf{x}_{B}, \mathbf{w}_{D}, \mathbf{a}_{D}^l)\), if there exists at least one pair \(p \in I(B)\) such that \(a_{p,p} = 1\). If \(X_B\) is complete with loop(s), denoted by \(X_B^{c,l}\), then \(a_{p,p} = 1\) for all \(p \in I(B)\).
	\end{enumerate}
\end{defn}

\begin{prop}
	The image under \(\varphi\) of:
	
	\begin{enumerate}
		\item  the empty graph \(\emptyset_{\mathbb{G}} \in \mathbb{G}\) is the empty element \(\emptyset_{\mathcal{D}}\), that is, \(\varphi(\emptyset_{\mathbb{G}}) = \emptyset_{\mathcal{D}}\);
		\item and (2) the total, complete, and null graph \(\mathcal{G}_{t,c}^{0} \in \mathbb{G}\) is the total, complete, and null element \(X_{\mathcal{B}_v}^{c,0}\), that is, \(\varphi(\mathcal{G}_{t,c}^{0}) = X_{\mathcal{B}_v}^{c,0}\).
	\end{enumerate}
	
	   In the case where loops are considered, the total, complete, and null graph with loops \(\mathcal{G}_{t,c}^{l,0} \in \mathbb{G}\) has for image the total, complete, and null element with loops \(X_{\mathcal{B}_v}^{c,l,0}\), that is, \(\varphi(\mathcal{G}_{t,c}^{l,0}) = X_{\mathcal{B}_v}^{c,l,0}\).
\end{prop}

\begin{proof}
	The results of this proposition are trivial; they directly follow from the application of the definition of \( \varphi \) in \eqref{mapGtoD}.
\end{proof}

\subsection{\label{AlgVBS}Algebraic Operations in Variable-Basis Spaces}
\noindent In this subsection, the operations corresponding to the additive union $\cup_+$  (Definition~\ref{AddUnionDef}), the additive intersection $\cap_+$  (Definition~\ref{AddInterDef}), and the multiplicative action law $\scal$  (Definition~\ref{ActionLaw}) are determined when the space under consideration  is the variable-basis space $\mathcal{D}_{\mathbb{G}}$.  To identify these operations, $\varphi$ is required to be a semi-linear map  from $\mathbb{G}$ to $\mathcal{D}_{\mathbb{G}}$.

\begin{defn}(Adapted from Def\textbf{.1.22} in \cite{janyvska2007semi})
	\label{semi-linear}
	The map \(\varphi: \mathbb{G} \to \mathcal{D}_{\mathbb{G}}\subset \mathcal{D}\) is said to be semi-linear between semi-vector spaces if, for every operation \(\star \in \{\cup_+, \cap_+\}\) in \(\mathbb{G}\), there exists an operation \(\diamond\) in \(\mathcal{D}_{\mathbb{G}}\) such that for all \(G, H \in \mathbb{G}\) and \(\lambda \in \mathbb{R}\), the following hold:
	\begin{enumerate}
		\item \(\varphi(\lambda \scal G) = \lambda \varphi(G)\);
		\item \(\varphi(G \star H) = \varphi(G) \diamond \varphi(H)\).
	\end{enumerate}   
\end{defn}

The scalar multiplication and the operation $\diamond$ in Definition~\ref{semi-linear}  must be consistent with their counterparts in $\mathbb{G}$.  In this space, scalar multiplication modifies only vertex attributes and edge weights,  while preserving the structure encoded in the adjacency matrix.  Accordingly, the definition of scalar multiplication is as follows.

\begin{defn}
	Scalar multiplication by \(\lambda \in \mathbb{R}\) of an element \((\mathbf{x}_{B}, \mathbf{w}_{D}, \mathbf{a}_{D}) \in \mathcal{D}\) is an external operation, defined as follows:
	\vspace{-2 mm}
	\[
	\lambda (\mathbf{x}_{B}, \mathbf{w}_{D}, \mathbf{a}_{D}) = (\lambda \mathbf{x}_{B}, \lambda \mathbf{w}_{D}, \mathbf{a}_{D}),
	\]
	where the components are defined as follows:
	\vspace{-1 mm}
	\[
	\lambda \mathbf{x}_{B} = \sum_{i \in I(B)} \lambda x_i \mathbf{b}_i, \quad \lambda \mathbf{w}_{D} = \sum_{(p,q) \in I(D)} \lambda w_{p,q} \mathbf{b}_p \otimes \mathbf{b}_q.
	\]
\end{defn}

To determine the operations \(\diamond\) mentioned in Definition \ref{semi-linear}, which are equivalent to \(\star \in \{\cup_+, \cap_+\}\), let us consider two graphs from \(\mathbb{G}\), \(X = (V_X, E_X)\) and \(Y = (V_Y, E_Y)\), whose images under \(\varphi\) are \(\varphi(X) = ( \mathbf{x}_{B},  \mathbf{w}_{D},  \mathbf{a}_{D})\) and \(\varphi(Y) = ( \mathbf{y}_{C},  \mathbf{v}_{L},  \mathbf{e}_{L})\). We seek the operations \(\diamond\)  for all $\star \in \{\cup_+, \cap_+\}$ such that:
\begin{equation}
	\label{Statment-Op-in-V-W-A}
	\varphi(V_X \star V_Y, E_X \star E_Y) \! = \!( \mathbf{x}_{B} \diamond  \mathbf{y}_{C},  \mathbf{w}_{D} \diamond  \mathbf{v}_{L},  \mathbf{a}_{D} \vartriangle  \mathbf{e}_{L}),
\end{equation}
where \(\vartriangle\) is a variant of \(\diamond\) in the Boolean case. In the following, we define operations \(\diamond\) that are equivalent to \(\star \in \{\cup_+, \cap_+\}\) in the real-valued spaces \(\mathcal{V}\) and \(\mathcal{W}\), as well as its variant \(\vartriangle\) in the Boolean-valued space \(\mathcal{A}\).

\subsubsection{\label{RealValueOp}Operations in  Variable-Basis Spaces $\mathcal{V}$ and $\mathcal{W}$}
\noindent In the spaces $\mathcal{V}$ and $\mathcal{W}$, the operations equivalent to $\cup_+$  and $\cap_+$ are denoted by $+_\cup$ and $+_\cap$, respectively.  They are illustrated for the space $\mathcal{V}$ in Fig.~\ref{SommeCross}  and defined as follows.

\begin{defn} The operation \(+_\star\) on \(\mathcal{V}\), with \(\star \in \{\cup, \cap\}\), is a mapping \(+_\star : \mathcal{V} \times \mathcal{V} \to \mathcal{V}\); for any pair \((\mathbf{x}_{B}, \mathbf{y}_{C}) \in \mathcal{V}^2\), where

	\[
	\mathbf{x}_{B} = \sum_{k \in I(B)} x_k \mathbf{b}_k \quad \text{and} \quad \mathbf{y}_{C} = \sum_{l \in I(C)} y_l \mathbf{b}_l,
	\]
	the operation \(+_\star\) is given by: 
	\[ 
	\mathbf{x}_{B} +_{\star} \mathbf{y}_{C} = \mathbf{x}_{B \star C} + \mathbf{y}_{C \star B}, \quad \text{with } \star \in \{\cup, \cap\}, 
	\]
	where \(\mathbf{x}_{B \star C}\) and \(\mathbf{y}_{C \star B}\) denote the projections of \(\mathbf{x}_{B}\) and \(\mathbf{y}_{C}\) onto the basis \(B \star C = C \star B\). More explicitly, \[ \mathbf{x}_{B} +_{\star} \mathbf{y}_{C} = \sum_{p \in I(B \star C)} (x_p + y_p) \mathbf{b}_p, \quad \text{with } \star \in \{\cup, \cap\}. \] If \(\star = \cup\), then for any \(p \in I(B)\) such that \(p \notin I(C)\) (and vice versa), \(y_p = 0\) (or \(x_p = 0\), respectively).
	\end{defn}
	
	\begin{defn}
		The operation \(+_\star\) on \(\mathcal{W}\), with \(\star \in \{\cup, \cap\}\), is a mapping \(+_\star : \mathcal{W} \times \mathcal{W} \to \mathcal{W}\); for any pair \((\mathbf{w}_D, \mathbf{\omega}_L) \in \mathcal{W}^2\), where 
		\vspace{-1 mm}
		\[
		\mathbf{w}_D \!=\!\! \!\!\! \sum_{(k,l) \in I(D)} \!\!\!\!  \!\!  w_{k,l} \mathbf{b}_k \otimes \mathbf{b}_l, \quad \mathbf{\omega}_L \!=\!\! \!\! \!\!\!  \sum_{(m,n) \in I(L)} \!\! \!\!   \omega_{m,n} \mathbf{b}_m \otimes \mathbf{b}_n,
		\]

		the operation \(+_\star\) is given by:

		\[
		\mathbf{w}_D +_\star \mathbf{\omega}_L = \mathbf{w}_{D \star L} + \mathbf{\omega}_{L \star D}, \quad \text{with } \star \in \{\cup, \cap\},
		\]  
		where \(\mathbf{w}_{D \star L}\) and \(\mathbf{\omega}_{L \star D}\) denote the projections of \(\mathbf{w}_D\) and \(\mathbf{\omega}_L\) onto the basis \(D \star L = L \star D\), with \(\star \in \{\cup, \cap\}\). More precisely, this is given by:
	
		\[
		\mathbf{w}_D +_\star \mathbf{\omega}_L = \sum_{(p,q) \in I(D \star L)} (w_{p,q} + \omega_{p,q}) \mathbf{b}_p \otimes \mathbf{b}_q.
		\]  
		If \(\star = \cup\), then for any \((p,q) \in I(D)\) such that \((p,q) \notin I(L)\) (and vice versa),  \(\omega_{p,q} = 0\) (or \(w_{p,q} = 0\), respectively).  
	\end{defn}
	
	\begin{figure}[h]
		\centering
		\begin{subfigure}[b]{.5\textwidth}
			\centering
			  \resizebox{.91\textwidth}{!}{
  
   \begin{tikzpicture}
    \def \dx {26 mm}
    \def \dy {2 mm}
    \def \dh {13 mm}
    \def \dm {3 mm}
    
    \draw[rounded corners=10pt, thick] (-\dx-3.5*\dy,-4.5*\dm) rectangle (-\dx + 8.3*\dh,4.5*\dm);
    \node at (-\dx,0) {
    $\begin{bmatrix}
    x_2 \\
    x_3  \\
    \end{bmatrix}$
    };
    
    \node at (-\dx +\dm, 2. mm)[right]{
        \textcolor{cyan}{ $b_2$}
    };
    \node at (-\dx +\dm, -3.5 mm)[right]{
        \textcolor{cyan}{ $b_3$}
    };
    
    \node at (-\dx +\dh,0) {
      $+_\cup$
    };
    
    
    \node at (-\dx + 1.7*\dh,0) {
    $\begin{bmatrix}
    y_1 \\
    y_2  \\
    y_3  \\
    y_4
    \end{bmatrix}$
    };
    
    \node at (-\dx + 1.7*\dh+\dm, 7.9 mm)[right]{
        \textcolor{brown}{ $b_1$}
    };
    
    \node at (-\dx + 1.7*\dh+\dm, 2. mm)[right]{
        \textcolor{cyan}{ $b_2$}
    };
    
    \node at (-\dx + 1.7*\dh+\dm, -3.5 mm)[right]{
        \textcolor{cyan}{ $b_3$}
    };
    
    \node at (-\dx + 1.7*\dh+\dm, -9.3 mm)[right]{
        \textcolor{brown}{ $b_4$}
    };
    
    \node at (-\dx + 2.8*\dh,0){
      $=$
    };
    
    \node at (-\dx + 3.6*\dh,0) {
    $\begin{bmatrix}
    0   \\
    x_2 \\
    x_3  \\
    0
    \end{bmatrix}$
    };
    
    \node at (-\dx + 3.6*\dh+\dm, 7.9  mm)[right]{
        \textcolor{cyan}{ $b_1$}
    };
    
    \node at (-\dx + 3.6*\dh+\dm, 2. mm)[right]{
        \textcolor{cyan}{ $b_2$}
    };
    
    \node at (-\dx + 3.6*\dh+\dm, -3.5 mm)[right]{
        \textcolor{cyan}{ $b_3$}
    };
    
    \node at (-\dx + 3.6*\dh+\dm, -9.3mm)[right]{
        \textcolor{cyan}{ $b_4$}
    };
    
    \node at (-\dx +4.6*\dh,0) {
      $+$
    };
    
    
    \node at (-\dx + 5.4*\dh,0) {
    $\begin{bmatrix}
    y_1 \\
    y_2  \\
    y_3  \\
    y_4
    \end{bmatrix}$
    };
    
    \node at (-\dx + 5.4*\dh+\dm, 7.9 mm)[right]{
        \textcolor{cyan}{ $b_1$}
    };
    
    \node at (-\dx + 5.4*\dh+\dm, 2. mm)[right]{
        \textcolor{cyan}{ $b_2$}
    };
    
    \node at (-\dx + 5.4*\dh+\dm, -3.5 mm)[right]{
        \textcolor{cyan}{ $b_3$}
    };
    
    \node at (-\dx + 5.4*\dh+\dm, -9.3 mm)[right]{
        \textcolor{cyan}{ $b_4$}
    };
    
    \node at (-\dx + 6.4*\dh,0){
      $=$
    };
    
     \node at (-\dx + 7.4*\dh,0) {
    $\begin{bmatrix}
    y_1 \\
    x_2 + y_2  \\
    x_3 + y_3  \\
    y_4
    \end{bmatrix}$
    };
     
  \end{tikzpicture}
  
  }
			\caption{\label{SommeUnion}}
		\end{subfigure}
		
		\begin{subfigure}[b]{.5\textwidth}
			\centering
			\resizebox{.91\textwidth}{!}{
  
   \begin{tikzpicture}
    \def \dx {26 mm}
    \def \dy {2 mm}
    \def \dh {13 mm}
    \def \dm {3 mm}
    
    \draw[rounded corners=10pt, thick] (-\dx-3.5*\dy,-4.5*\dm) rectangle (-\dx + 8.3*\dh,4.5*\dm);
    
    \node at (-\dx,0) {
    $\begin{bmatrix}
    x_2 \\
    x_3  
    \end{bmatrix}$
    };
    
    \node at (-\dx +\dm, 2. mm)[right]{
        \textcolor{cyan}{ $b_2$}
    };
    \node at (-\dx +\dm, -3.5 mm)[right]{
        \textcolor{cyan}{ $b_3$}
    };
    
    \node at (-\dx +\dh,0) {
      $+_\cap$
    };
    
    
    \node at (-\dx + 1.7*\dh,0) {
    $\begin{bmatrix}
    y_1 \\
    y_2  \\
    y_3  \\
    y_4
    \end{bmatrix}$
    };
    
    \node at (-\dx + 1.7*\dh+\dm, 7.9 mm)[right]{
        \textcolor{brown}{ $b_1$}
    };
    
    \node at (-\dx + 1.7*\dh+\dm, 2. mm)[right]{
        \textcolor{cyan}{ $b_2$}
    };
    
    \node at (-\dx + 1.7*\dh+\dm, -3.5 mm)[right]{
        \textcolor{cyan}{ $b_3$}
    };
    
    \node at (-\dx + 1.7*\dh+\dm, -9.3 mm)[right]{
        \textcolor{brown}{ $b_4$}
    };
    
    \node at (-\dx + 2.8*\dh,0){
      $=$
    };
    
    \node at (-\dx + 3.6*\dh,0) {
    $\begin{bmatrix}
    x_2 \\
    x_3  
    \end{bmatrix}$
    };
    
    \node at (-\dx + 3.6*\dh+\dm, 2. mm)[right]{
        \textcolor{cyan}{ $b_2$}
    };
    
    \node at (-\dx + 3.6*\dh+\dm, -3.5 mm)[right]{
        \textcolor{cyan}{ $b_3$}
    };

    \node at (-\dx +4.6*\dh,0) {
      $+$
    };
    
    
    \node at (-\dx + 5.4*\dh,0) {
    $\begin{bmatrix}
    y_2  \\
    y_3  
    \end{bmatrix}$
    };

    \node at (-\dx + 5.4*\dh+\dm, 2. mm)[right]{
        \textcolor{cyan}{ $b_2$}
    };
    
    \node at (-\dx + 5.4*\dh+\dm, -3.5 mm)[right]{
        \textcolor{cyan}{ $b_3$}
    };
    
    \node at (-\dx + 6.4*\dh,0){
      $=$
    };
    
     \node at (-\dx + 7.4*\dh,0) {
    $\begin{bmatrix}
    x_2 + y_2  \\
    x_3 + y_3  
    \end{bmatrix}$
    };
     
  \end{tikzpicture}
  
  }
			\caption{\label{SommInter}}
		\end{subfigure}
		
		\caption{\label{SommeCross}Examples illustrating sums over (a) the unified basis and (b) the maximal common basis of the operands in the variable-basis real-valued space of states \(\mathcal{V}\).}
	\end{figure}
	
	\subsubsection{\label{BoolValueOp}Operations in the Variable-Basis  Space $\mathcal{A}$}
	In the Boolean-valued space \(\mathcal{A}\), the variant \(\vartriangle\) of \(\diamond\) in equation \eqref{Statment-Op-in-V-W-A} must satisfy the axioms of Boolean algebra. To achieve this, the overlapping operations \(\cup_+\) and \(\cap_+\) induce the corresponding operations \(\vee_\cup\) and \(\wedge_\cap\) in \(\mathcal{A}\), where \(\vee\) and \(\wedge\) denote the disjunction and conjunction operations, respectively. These operations follow the fundamental axioms of Boolean algebra \cite{halmos2009introduction}, with the most relevant ones for this paper given as follows: for all \(p \in \{0,1\}\),  
	\[
	p \vee 1 = 1, \quad p \vee 0 = p, \quad p \wedge 1 = p, \quad p \wedge 0 = 0.
	\]  
	Additionally, both disjunction (\(\vee\)) and conjunction (\(\wedge\)) are commutative operations.
	
	\begin{defn}
		The operation \(\star \in \{\vee_\cup, \wedge_\cap\}\) on \(\mathcal{A}\) is a mapping \(\star : \mathcal{A} \times \mathcal{A} \to \mathcal{A}\); for any \((\boldsymbol{\gamma}_{D}, \boldsymbol{\delta}_{L} )\in \mathcal{A}^2\), such that  
		\vspace{-2 mm}\[
		\boldsymbol{\gamma}_{D} = \sum_{(i,j) \in I(D)} \gamma_{i,j} \mathbf{b}_i \otimes \mathbf{b}_j, \quad \boldsymbol{\delta}_{L} = \sum_{(k,l) \in I(L)} \delta_{k,l} \mathbf{b}_k \otimes \mathbf{b}_l,
		\]  
		
		\begin{itemize}  
			\item the disjunction $(\vee_\cup)$ is given by:  
			\[
			\boldsymbol{\gamma}_{D} \vee_\cup \boldsymbol{\delta}_{L} = \sum_{(p,q) \in I(D \cup L)} (\gamma_{p,q} \vee \delta_{p,q}) \mathbf{b}_p \otimes \mathbf{b}_q.
			\]  
			When, for any \( (p,q) \in I(D) \) such that \( (p,q) \notin I(L) \) (and vice versa), \( \delta_{p,q} = 0 \) (or \( \gamma_{p,q} = 0 \), respectively).
			
			\vspace{1 mm}
			
			\item the conjunction $(\wedge_\cap)$ is given by: 
			\[
			\boldsymbol{\gamma}_{D} \wedge_\cap \boldsymbol{\delta}_{L} = \sum_{(p,q) \in I(D \cap L)} (\gamma_{p,q} \wedge \delta_{p,q}) \mathbf{b}_p \otimes \mathbf{b}_q.
			\]  
		\end{itemize} 
	\end{defn}
	
	\subsection{\label{AlgbStruc}Algebraic Structures in the Variable-Basis Space}
	\noindent Quite naturally, through the isomorphism \(\varphi\) (see equation \eqref{mapGtoD}, Proposition \ref{Bijectionphi}, and Definition \ref{semi-linear}), the space \(\mathcal{D}_{\mathbb{G}}\) inherits the properties of the space \(\mathbb{G}\) (see Subsection \ref{MainResult}). Therefore, we establish the following corollaries:
	
	\begin{cor}
		\label{FreedimavecCup}
		The structure \((\mathcal{D}_{\mathbb{G}}, +_\cup, \emptyset_{\mathcal{D}},\scal)\) forms a simple semi-vector space over the field \(\mathbb{R}\).
	\end{cor}
	
	\begin{cor}
		\label{FreedimavecCap}
		The structure \((\mathcal{D}_{\mathbb{G}}, +_\cap, X^{c,0}_{\mathcal{B}_v},\scal)\) forms a simple semi-vector space over the field \(\mathbb{R}\).
	\end{cor}
	
	\begin{proof}
		The proofs of corollaries \ref{FreedimavecCup} and \ref{FreedimavecCap} rely on Proposition \textbf{3.4} from \cite{la2021semi}, which states that the image of a semi-vector space under a semi-linear mapping is itself a semi-vector subspace. Specifically, \((\mathcal{D}_{\mathbb{G}}, +_\cup, \emptyset_{\mathcal{D}})\) and \((\mathcal{D}_{\mathbb{G}}, +_\cap, X^{c,0}_{\mathcal{B}_v})\) are the images of the semi-vector spaces \((\mathbb{G}, \cup_+, \emptyset_{\mathbb{G}})\) and \((\mathbb{G}, \cap_+, \mathcal{G}_{t,c}^0)\), respectively, under \(\varphi\). 
	\end{proof}
	
	\begin{cor}
		\label{IsoStruct}
		The structure \((\mathcal{D}_{\mathbb{G}}, +_\cup, \emptyset_{\mathcal{D}}, +_\cap, X^{c,0}_{\mathcal{B}_v},\scal)\) is isomorphic to \((\mathbb{G}, \cup_+, \emptyset_{\mathbb{G}}, \cap_+, \mathcal{G}_{t,c}^0,\scal)\).
	\end{cor}
	
	\begin{proof}
		Crollary \ref{IsoStruct} is based on the fact that the semi-linear map \(\varphi\) is an isomorphism between \(\mathbb{G}\) and \(\mathcal{D}_\mathbb{G}\), as it is bijective (see Proposition \ref{Bijectionphi}).
	\end{proof}
	
	In the case where loops are included, it can be stated that \((\mathcal{D}_{\mathbb{G}}, +_\cup, \emptyset_{\mathcal{D}}, +_\cap, X^{c,l,0}_{\mathcal{B}_v})\) and \((\mathbb{G}, \cup_+, \emptyset_{\mathbb{G}}, \cap_+, \mathcal{G}_{t,c}^{l,0})\) are isomorphic. As additional results, we also have that the sets \(\mathcal{V}\) and \(\mathcal{W}\), endowed with either \(+_\cup\) or \(+_\cap\), are semi-vector spaces over \(\mathbb{R}\).

\section{\label{HybDynModel}Hybrid Dynamics Model}
\noindent The modeling of systems whose state is described by a graph naturally falls within the framework of hybrid dynamical systems, as these systems combine continuous and discrete dynamics (cf. Fig.~\ref{GrapDynam}). Among the various existing formalisms (hybrid automata, impulsive differential equations, switching systems), we adopt the one proposed by Goebel et al. \cite{goebel2012hybrid} and formulate this model in the variable-basis space $\mathcal{D}_{\mathbb{G}}$.

\subsection{Basic Definitions and Notation}
\noindent Since $\mathcal{D}_{\mathbb{G}}$ is a newly introduced space here, we define and assign notations to the concepts of hybrid system state, hybrid arc, and hybrid disturbance over the following hybrid time domain:

\begin{defn}
	A subset $\mathcal{T} \subset \mathbb{R}_{\ge 0} \times \mathbb{N}$ is a compact hybrid time domain if

	\[
	\mathcal{T} := \bigcup_{k = 0}^{K-1} \big([\tau_k, \tau_{k+1}], k\big)
	\]
	for a finite sequence of times $0=\tau_0 \le \tau_1 \le \dots \le \tau_K$.  The subset $\mathcal{T} \subset \mathbb{R}_{\ge 0} \times \mathbb{N}$ is a hybrid time domain if it is the union of a nondecreasing sequence of compact hybrid time domains, that is, if $\mathcal{T}$ is the union of compact hybrid time domains $\mathcal{T}_k$ satisfying:  
	\[
	\mathcal{T}_0 \subset \mathcal{T}_1 \subset \dots \subset \mathcal{T}_k \subset \dots
	\]
\end{defn}

\begin{defn}
	\label{StateHyb}
	The state of the hybrid system in $\mathcal{D}_{\mathbb{G}}$ is denoted by $X_{B(k)}(t, k)$ and is defined as
	\[
	X_{B(k)}(t, k) = (\mathbf{x}_{B(k)}(t, k), \mathbf{w}_{D(k)}(t, k), \mathbf{a}_{D(k)}(t,k)),
	\]
	where \(B(k) \subseteq \mathcal{B}_v\) and \(D(k) = B(k) \otimes B(k)\) are the bases associated with the \(k\)-th state jump. For simplicity, we denote the state as
	\[
	X_{B_k}(t) := X_{B(k)}(t, k),
	\]
	with components denoted by $\mathbf{x}_{B_k}(t) := \mathbf{x}_{B(k)}(t, k)$, $\mathbf{w}_{D_k}(t) := \mathbf{w}_{D(k)}(t, k)$, and $\mathbf{a}_{D_k}(t) := \mathbf{a}_{D(k)}(t,k)$.
\end{defn}

\begin{defn}
	\label{ArcHybrid}
	A function $\phi: \mathcal{T} \to \mathcal{D}_\mathbb{G}$ is a \emph{hybrid arc} if, for each $k \in \mathbb{N}$, the function $t \mapsto \phi(t, k)$ is absolutely continuous on the interval $\Theta_k := \{ t \mid (t,k) \in \mathcal{T}\}.$ For all $(t, k) \in \mathcal{T}$, there exists $Y_{B_k}(t) \in \mathcal{D}_{\mathbb{G}}$ such that  $Y_{B_k}(t) = \phi(t, k)$.
\end{defn}

\begin{defn}
	\label{HybridDisturbance}
	The hybrid disturbance is a mapping $\mu~:~\mathcal{T}\to\mathbb{R}^m\times~\mathcal{D}_{\mathbb{G}}$, such that

	\[
	\mu(t,k) = (\mathbf{u}(t), U_{C_k}(t)) \in \mathbb{R}^m \times \mathcal{D}_{\mathbb{G}},
	\]
	where \(C_k \in 2^{\mathcal{B}_v}\). For all \(k \in \mathbb{N}\), the function
	\(t \mapsto \mu(t,k)\) is Lebesgue measurable and locally essentially bounded on the interval
	\(\Theta_k = \{t \mid (t,k) \in \mathcal{T}\}\).
\end{defn}

\subsection{General Form of the Model}
\noindent The dynamics of a hybrid system are modeled using either equations or differential and difference inclusions under constraints. In this paper, we focus on equations, whose general form is:
\begin{equation}
	\label{GenEquaDiff}
	\left\{
	\begin{aligned}
		(X_{B_k}(t),&\mathbf{u}(t)) \in F,~~  \dot{X}_{B_k}(t)= f\left(X_{B_k}(t),\mathbf{u}(t)\right),\\
		\\
		(X_{B_k}(t),&U_{C_k}(t)) \in J, ~~  X_{B_{k+1}}(t) = j\left(X_{B_k}(t),U_{C_k}(t)\right),\\
		\\
		(X_0,&\mu_0)=(X_{B_0}(0),\mu(0,0)),
	\end{aligned}
	\right.
\end{equation}

where $F \subset \mathcal{D}_\mathbb{G} \times \mathbb{R}^m$ and $J \subset \mathcal{D}_\mathbb{G}^2$ are the flow set and jump set, respectively, while $f$ and $j$ are the flow map and jump map; these elements constitute the data of the hybrid system, denoted by $\mathcal{H} = (F,f,J,j)$, or simply $\mathcal{H}$. We set, for all \(B \in 2^{\mathcal{B}_v}\), $\mathcal{D}(B) = \mathcal{V}(B) \times \mathcal{W}(D) \times \mathcal{A}(D)$, with  $D = B \otimes B$. We assume that the flow map $f : \mathcal{D}(B) \times \mathbb{R}^m \to \mathcal{D}(B)$ is well-defined for all \(B \in 2^{\mathcal{B}_v}\) and governs exclusively the dynamics of the attributes and weights, namely those in \(\mathcal{V}(B)\) and \(\mathcal{W}(D)\), while the structure in \(\mathcal{A}(D)\) is discrete and does not evolve. Thus, the continuous dynamics in equation \eqref{GenEquaDiff} can be written as a coupled system of two equations as follows: if $(X_{B_k}(t), \mathbf{u}(t)) \in F$, then

\begin{equation}
	\label{CoupledDyn}
	\left\{
	\begin{aligned}
		\dot{\mathbf{x}}_{B_k}(t) &= f_v(\mathbf{x}_{B_k}(t),\mathbf{w}_{D_k}(t), \mathbf{a}_{D_k},\mathbf{u}(t)), \\
		\dot{\mathbf{w}}_{D_k}(t) &= f_e(\mathbf{x}_{B_k}(t),\mathbf{w}_{D_k}(t), \mathbf{a}_{D_k},\mathbf{u}(t)),\\
		\dot{\mathbf{a}}_{D_k}(t) &=  \mathbf{0}_D,
	\end{aligned}
	\right.
\end{equation}
where $\mathbf{0}_D$ is the matrix with all zero coordinates on the basis $D$. For any $B \in 2^{\mathcal{B}_v}$ and $D = B \otimes B$, the components of the flow map f are
\begin{equation}
	\label{Deflowmap}
	\begin{split}
		f_v &: \mathcal{D}(B)  \times \mathbb{R}^m  \to \mathcal{V}(B) , \\
		f_e &: \mathcal{D}(B) \times \mathbb{R}^m  \to \mathcal{W}(D).
	\end{split}
\end{equation}

The jump map $j$ is an application in charge of discrete transformations of the system state, acting simultaneously on the topology (adjacency matrix), the attributes, and the weights. It is defined as a composition law 

\begin{equation}
	\label{jumpmapform}
	j : \mathcal{D}_{\mathbb{G}}\times \mathcal{D}_{\mathbb{G}}\to \mathcal{D}_{\mathbb{G}}
\end{equation}
Assuming that the jump set is $J = J_{\nearrow} \cup J_{\searrow} \cup (J_{\star} \times \mathcal{V})$, whose components are defined in \eqref{jumpset}, we can distinguish three implementations of the jump map $j$ as follows:
\begin{equation}
	\label{Defjump}
	\begin{split}
		&j(X_{B_k}(t),U_{C_k}(t)) =\\
		&\begin{cases}
			X_{B_k}(t) +_\cup U_{C_k}(t), & \text{if } (X_{B_k}(t),U_{L_k}(t)) \in J_{\nearrow},\\
			X_{B_k}(t) +_\cap U_{C_k}(t), & \text{if } (X_{B_k}(t),U_{L_k}(t)) \in J_{\searrow},\\
			\tilde{j}(X_{B_k}(t)), & \text{if } X_{B_k}(t) \in J_\star.
		\end{cases}
	\end{split}
\end{equation}

The first implementation tends to increase the topology of $X_{B_k}$ by adding elements via the external perturbation $U_{C_k}$, whereas the second tends to reduce it by specifying in $U_{C_k}$ the topology to be preserved. Recall that the operations $+_\cup$ and $+_\cap$ (see Subsection \ref{AlgVBS}) also act on the attributes and weights. The last case corresponds to a discrete change induced by the system’s intrinsic dynamics. Assuming $J_\star = J_{+} \cup J_{-}$, whose definitions are given in \eqref{jumpset}, we can consider two implementations of this intrinsic discrete dynamics as follows:
\begin{equation}
	\label{autojump}
	\begin{split}
		\tilde{j}&(X_{B_k}(t)) =\\
		&\begin{cases}
			X_{B_k}(t) +_\cup g(X_{B_k}(t)), & \text{if } X_{B_k}(t) \in J_{+},\\
			X_{B_k}(t) +_\cap h(X_{B_k}(t)), & \text{if } X_{B_k}(t) \in J_{-}.
		\end{cases}
	\end{split}
\end{equation}
The function $g$ in \eqref{autojump} can be chosen in the form
\begin{equation}
	g(X_{B_k}(t)) = \xi^+ \scal \mathbf{1}(X_{B_k}(t)) +_\cup Y_L,
\end{equation}

where $\mathbf{1}(X_{B_k}(t))$ is the unit vector constructed from $X_{B_k}$, $\xi^+ \in \mathbb{R}$ represents the perturbation induced by the elements added via $Y_L \in \mathcal{D}_\mathbb{G}$, with $B_k \cap L = \varnothing$. In a stochastic framework, one may assume that $L \in 2^{\mathcal{B}_v}$ is associated with a conditional probability given $B_k \in 2^{\mathcal{B}_v}$, or even that $Y_L$ has a conditional probability \emph{given} $X_{B_k}$. The function $h$, in turn, can be defined as
\begin{equation}
	h(X_{B_k}(t)) = \xi^- \scal \mathbf{1}_H,
\end{equation}
where $\mathbf{1}_H \in \mathcal{D}_\mathbb{G}$ is a unit vector associated with $H \subset B_k$, and $\xi^-$ represents the perturbation induced by the removal of the elements $I(B_k) \setminus I(H)$. The subsets involved in the characterization of the jump set $J$ can be described as follows:
\begin{equation}
	\label{jumpset}
	\begin{aligned}
		J_{\nearrow} &= \left\{(X_{B},Y_{C}) \in \mathcal{D}_\mathbb{G}^2 \mid C\neq \varnothing, \exists H \subseteq C \text{ s.t. } H \not\subset B\right\}, \\
		J_{\searrow} &= \left\{(X_{B},Y_{C}) \in \mathcal{D}_\mathbb{G}^2 \mid C\neq \varnothing \text{ s.t. } C \subset B\right\}, \\
		J_{+} &= \left\{X_{B}\in \mathcal{D}_\mathbb{G} \mid \exists K \subseteq I(B)  \text{ s.t. }  x_k \geq \kappa > 0~~ \forall k \in K\right\}, \\
		J_{-} &= \left\{X_{B}\in \mathcal{D}_\mathbb{G} \mid \exists K \subseteq I(B)  \text{ s.t. }  0 < x_k \leq \lambda ~~ \forall k \in K\right\},
	\end{aligned}
\end{equation}
where $\kappa, \lambda \in \mathbb{R}$ and the $x_k$ are the attributes of index $k$ carried by the vector $X_{B_k}$. The function \(j\), defined as a composition law according to \eqref{jumpmapform}, essentially using the operations \(+_\star\) with \(\star \in \{\cup,\cap\}\) (see \eqref{Defjump} and \eqref{autojump}), is said to be closed/stable if and only if the operation \(+_\star\), with \(\star \in \{\cup,\cap\}\), is closed/stable.

\begin{prop}
	The set $\mathcal{D}_{\mathbb{G}} \subset \mathcal{D}$ is closed (or stable) under the operation $+_\star$, where $\star \in \{\cup,\cap\}$.
\end{prop}

\begin{proof}
	According to Bourbaki \cite{bourbaki1998algebra}, a subset $\mathcal{D}_{\mathbb{G}} \subset \mathcal{D}$ is said to be stable under a binary operation $\star$ if, for all $X, Y \in \mathcal{D}_{\mathbb{G}}$, we have $X \star Y \in \mathcal{D}_{\mathbb{G}}$. Since the set of graphs $\mathbb{G}$ is stable under $\cup_+$ and $\cap_+$ by definition (see Definitions~\ref{AddUnionDef} and~\ref{AddInterDef}), and since $(\mathcal{D}_{\mathbb{G}}, +_\cup, +_\cap)$ inherits this structure through the semi-linear mapping $\varphi$ (see Corollary~\ref{IsoStruct}), it follows that $\mathcal{D}_{\mathbb{G}} \subset \mathcal{D}$ is closed (or stable) under $+_\star$, with $\star \in \{\cup,\cap\}$.
\end{proof}

\subsection{\label{ResExist}Existence of Solutions}
\noindent The existence of solutions is one of the criteria for well-posedness of the problem. Its proof requires specifying the notion of solution, i.e., the class of functions in which solutions are sought \cite{ccamhbel2002solution}. Here, we adopt the following notion of solution:

\begin{defn}
\label{DefSolution}
A pair $(\phi,\mu)$, where $\phi$ is a hybrid arc (see Definition~\ref{ArcHybrid}) and $\mu = (\mathbf u, U_{C_k})$ is a hybrid disturbance (see Definition~\ref{HybridDisturbance}),  is a solution pair to the hybrid system $\mathcal H$ modeled in \eqref{GenEquaDiff} if $(\phi(0,0),\mathbf{u}(0)) \in F$ or $(\phi(0,0),U_0(0)) \in J$, and the following conditions hold:

\begin{enumerate}
	\item For all $k \in \mathbb N$ such that $\Theta_k \neq \emptyset$,
	\vspace{- 1 mm}
	\[
	(\phi(t,k),\mathbf{u}(t)) \in F
	\quad \text{for all } t \in \operatorname{int} \Theta_k,
	\]
	\vspace{- 2.5 mm}
	\[
	\frac{d}{dt} \phi(t,k) = f(\phi(t,k),\mathbf{u}(t))
	\quad \text{for almost all } t \in \Theta_k.
	\]
	
	\item For all $(\tau_k,k) \in \mathcal{T}$ such that $(\tau_{k},k+1) \in \mathcal{T}$,
	
		\vspace{- 1 mm}
	
	\[
	(\phi(\tau_k,k),U_{C_k}(\tau_k)) \in J,
	\]
	
	\vspace{-2.5 mm}
	
	\[
	\phi(\tau_{k},k+1) = j(\phi(\tau_k,k),U_{C_k}(\tau_k)).
	\]
\end{enumerate}

A solution pair $(\phi,\mu)$ is classified as: i) nontrivial if $\mathcal{T}$ has at least two points; ii) complete if $\mathcal{T}$ is unbounded; iii) maximal if there does not exist another solution pair $(\phi',\mu')$ such that $(\phi,\mu)$ is a truncation of $(\phi',\mu')$ on a proper subset of $\mathcal{T}$.
\end{defn}

To guarantee the existence of solutions on each interval $\Theta_k$, for all $k \in \mathbb{N}$, we require that the flow map $f$ satisfy the Carathéodory conditions \cite{filippov2013differential,biles1997technique}, which are:

\begin{assump}
	\label{CaratheodoryAssum}
	Let $B \subset \mathcal B_v$, define $D = B \otimes B$, and set $ \mathcal D(B) = \mathcal V(B) \times \mathcal W(D) \times \mathcal A(D)$.  We assume that the function  $f : \mathcal D(B) \times \mathbb{R}^m \to \mathcal D(B)$ satisfies the following properties:
	
	\begin{enumerate}
		\item Measurability in time: For every fixed $X_B \in \mathcal D(B)$, the mapping  $t \longmapsto f(X_B, \mathbf u(t))$ is measurable on $\Theta_k$.
		
		\item Continuity in space (locally): For almost every $t \in \Theta_k$, the mapping $X_B \longmapsto f(X_B, \mathbf u(t)) $ is continuous on every compact $K \subset \mathcal D(B)$.
		
		\item Local boundedness: For every compact $K \subset \mathcal D(B)$, there exists a function $m_K \in L^1_{\mathrm{loc}}(\Theta_k)$ such that   $\| f(X_B, \mathbf u(t)) \| \le m_K(t)$,  $\forall X_B \in K$, for almost every  $t \in \Theta_k$.
	\end{enumerate}
\end{assump}

\begin{note}
	The second point in Assumption \ref{CaratheodoryAssum} only makes sense for two reasons. First, $\mathcal D(B)$ has a fixed size, determined by the base $B$. Second, the Boolean component of an element $X_B \in \mathcal D(B)$, which belongs to $\mathcal A(D)$, is constant for all $t \in \Theta_k$, for every $k \in \mathbb{N}$.
\end{note}

\begin{thm}
	Consider the hybrid system $\mathcal{H}$ in \eqref{GenEquaDiff}. 
	Assume that, for every $B \in 2^{\mathcal{B}_v}$, the flow map $f$ is well defined and satisfies Assumption~\ref{CaratheodoryAssum}.  Assume moreover that the jump map $j$, given by \eqref{jumpmapform}, satisfies $j\bigl(\mathcal{D}_{\mathbb{G}}~\times~\mathcal{D}_{\mathbb{G}}\bigr)~\subseteq~\mathcal{D}_{\mathbb{G}}$. Then, there exists a pair of solutions $(\phi,\mu)$ to the problem $\mathcal{H}$ in the sense of Definition~\ref{DefSolution}.
\end{thm}

\begin{proof}
	The construction of the solution $(\phi, \mu)$, where $\phi$ and $\mu$ are respectively a hybrid arc and a hybrid disturbance (as defined in Definition~\ref{ArcHybrid} and Definition~\ref{HybridDisturbance}), is performed by induction on the jump index $k \in \mathbb{N}$.
	
	Initialization ($k=0$): Let $\phi(0,0) = \xi_0 \in \mathcal{D}_{\mathbb{G}}$ and $\mu(0,0) = (\eta_0, \zeta_0) \in \mathbb{R}^m \times \mathcal{D}_{\mathbb{G}}$.
	
	\vspace{-1 mm}
	
	\begin{itemize}
		\item  If $(\xi_0, \eta_0) \in F$, condition (1) of Definition~\ref{DefSolution} is activated. Assumption~\ref{CaratheodoryAssum} on the flow map $f$ guarantees the existence of an absolutely continuous mapping $t \mapsto \phi(t,0)$ on an interval $\Theta_0$ \cite{biles1997technique}. The measurability of $\mathbf{u}(t)$ is a direct consequence of the composition principle: for $t \mapsto f(x, \mathbf{u}(t))$ to be Lebesgue-measurable for a fixed $x$ (per Assumption~\ref{CaratheodoryAssum}), it is necessary that $\mathbf{u}(t)$ be measurable.
		\item If $(\xi_0, \zeta_0) \in J$, condition (2) of Definition~\ref{DefSolution} is activated. By the  hypothesis $j(\mathcal{D}_{\mathbb{G}} \times \mathcal{D}_{\mathbb{G}}) \subseteq \mathcal{D}_{\mathbb{G}}$, the new state $\phi(0,1) = j(\xi_0, \zeta_0)$ remains in $\mathcal{D}_{\mathbb{G}}$.
	\end{itemize}
	
	In both cases, the solution is always non-trivial as it can be extended beyond the initial instant $(0,0)$, either through a continue time interval or a jump to an admissible state.  Induction: For any $(t,k) \in \mathcal{T}$, as long as the state remains in $F \cup J$:
	
	\vspace{-1 mm}
	\begin{itemize}
		\item  If $(\phi(t,k), \mathbf{u}(t)) \in F$, the solution is extended over the interval $\Theta_k$ via Assumption~\ref{CaratheodoryAssum}, ensuring $t \mapsto \phi(t,k)$ is absolutely continuous and $t \mapsto \mathbf{u}(t)$ is Lebesgue-measurable.
		\item If $(\phi(t,k), U_{C_k}(t)) \in J$, a jump is triggered such that $\phi(t, k+1) = j(\phi(t,k), U_{C_k}(t))$. The state $\phi(t, k+1)$ can be an admissible starting point for a subsequent flow or jump phase.
	\end{itemize}
	
	\vspace{-1 mm}
	
	The solution is complete if this process continues indefinitely (i.e., $t+k \to \infty$). If the state leaves $F \cup J$, the solution is maximal and exists over a bounded hybrid domain $\mathcal{T}_{\text{max}} \subset \mathcal{T}$.
\end{proof}
\section{\label{MicrobioteApp}Application to Gut Microbiota Dynamics}
\noindent  The objective of this section is to illustrate how the previously developed framework can be used to model and simulate systems that combine structural changes with continuous variations of quantitative variables. A particularly relevant application is the dynamics of the gut microbiota under antibiotic treatment, followed by bacteriotherapy. This application primarily serves to resolve the difficulty raised by Hirsch \cite{hirsch1984dynamical} regarding fixed-structure state-spaces, particularly within the field of theoretical ecology.

The gut microbiota is a complex and dynamic community of microorganisms inhabiting the gastrointestinal tract. It plays a central role in immune and metabolic homeostasis, as well as in protecting the host against pathogens \cite{thursby2017introduction}. These microorganisms engage in intricate interactions that sustain their growth and stability \cite{lin2022linear}. However, external perturbations, such as antibiotic treatments, can disrupt both the composition of the microbiota and the interactions within it. This disruption, referred to as dysbiosis, is associated with various conditions, including obesity, inflammatory diseases, and HIV infection \cite{thursby2017introduction, lin2022linear}.

In this context, bacteriotherapy has emerged as a therapeutic strategy aimed at restoring microbiota balance by selectively introducing or removing microbial species to achieve a desired equilibrium. Jones et al. \cite{jones2020navigation} modeled microbiota dynamics under antibiotic administration followed by bacteriotherapy using the generalized Lotka–Volterra (gLV) model \cite{taylor1988consistent}. In a previous study \cite{doliveira2026generalized}, we proposed implementing this model in the variable-base space \(\mathcal{V}\) (see Definition \ref{VBS_attributes}), with constant interaction weights, in order to distinguish zero abundance (\(x_i = 0\)) from the actual absence of species \(i\), as discussed below. Furthermore, implementing the dynamics in \(\mathcal{D}_{\mathbb{G}}\), which is isomorphic to the graph state space \(\mathbb{G}\), allows for the simultaneous evolution of topology, population abundances, and interaction weights. This formulation requires specifying the data of the hybrid system \(\mathcal{H}\) defined by \eqref{GenEquaDiff}–\eqref{jumpset}. More precisely, the function \(f_v\) in \eqref{CoupledDyn} describes the gLV dynamics of microbial abundances under antibiotic treatment, while the function \(j\) in \eqref{Defjump} models bacteriotherapy. Its variant \(\tilde{j}\), defined in \eqref{autojump}, accounts for the appearance or disappearance of species when their abundances cross a prescribed threshold.

In contrast, no model in the literature describes the dynamics of interaction weights \cite{lin2022linear}. To complete the specification of the hybrid system \(\mathcal{H}\), we adopt here an illustrative approach by defining the function \(f_e\) in \eqref{CoupledDyn} based on models of synaptic weight dynamics in neuronal circuits; more specifically, we adapt Oja’s rule into a linear formulation (see equation 8.16 in \cite{dayan2005theoretical}). With this choice, the coupled flow system \eqref{CoupledDyn} can be written as follows:

\begin{equation}
	\label{flowModel}
	\left\{
	\begin{array}{rlc}
		\dot{\mathbf{x}}_{B_k}(t) &=& \mathbf{x}_{B_k}(t) \circ\left\{\boldsymbol{\rho}_{B_k} + \mathbf{a}_{D_K}\circ\mathbf{w}_{D_K}(t) \mathbf{x}_{B_k}(t)\right. \\
		& & \left. + u(t) \scal \boldsymbol{\varepsilon}_{B_k}\right\},\\
		\\
		\dot{\mathbf{w}}_{D_k}(t) &=& \alpha \scal   \mathbf{w}_{D_K}(t) \circ \mathbf{a}_{D_k} + \beta \scal \mathbf{x}_{B_k}(t) \mathbf{x}_{B_k}(t)^{\text{T}},
	\end{array}
	\right.
\end{equation}
where \(\circ\) denotes the Hadamard product (element-wise multiplication).  \(\mathbf{x}_{B_k}\) is the vector of abundances \(x_i\) for populations \(i \in I(B_k)\), and \(\mathbf{x}_{B_k}^{\text{T}}\) its transpose.   The vectors \(\boldsymbol{\rho}_{B_k}\) and \(\boldsymbol{\varepsilon}_{B_k}\) represent, respectively, the growth rates and susceptibilities of the same species. The effect of the antibiotic is modeled by the control function
\[
u(t) =
\begin{cases}
	1, & t_0 < t < t_\star, \\
	0, & t > t_\star.
\end{cases}
\]
The scalar parameters are given by \(\alpha = -2 \times 10^{-2}\) and \(\beta = -1 \times 10^{-1}\). We fix the universe \(\mathcal{I}\) as the set of the 11 species listed by Stein et al. in Figure 2 of \cite{stein2013ecological}. Consequently, for all \(k \in \mathbb{N}\), we have \(I(B_k) \subset \mathcal{I}\). The growth rates and susceptibilities are also taken from this figure. We set \(\lambda = 1 \cdot 10^{-6}\) in the set \(J_-\) defined in \eqref{jumpset}, and choose \(t_\star = 4\) days. The initial state is given by \(X_0 = (\mathbf{x}_{B_0}, \mathbf{w}_{D_0}, \mathbf{a}_{D_0})\), where \(B_0 = \{\mathbf{b}_1, \mathbf{b}_2, \mathbf{b}_4, \mathbf{b}_5\}\) and \(D_0 = B_0 \otimes B_0\), with
\[
\mathbf{x}_{B_0}
= 0.7 \mathbf{b}_1 + 0.3 \mathbf{b}_2 + 1.2 \mathbf{b}_4 + 1.3 \mathbf{b}_5,
\]

\[
\mathbf{w}_{D_0} =
\begin{array}{c c c c|c}
	\mathbf{b}_1 & \mathbf{b}_2 & \mathbf{b}_4 & \mathbf{b}_5 & \\ \cline{1-4}
	-0.21 & 0.1 & -0.16 & -0.014 & \mathbf{b}_1 \\
	0.06 & -0.1 & -0.15 & -0.19 & \mathbf{b}_2 \\
	0.22 & 0.14 & -0.83 & -0.22& \mathbf{b}_4 \\
	-0.18 & 0& -0.05 & -0.51& \mathbf{b}_5
\end{array}
\]

\[
\mathbf{a}_{D_0} = (a_{p,q}), \quad \text{with } a_{p,q} = 1 \ \forall p,q \in I(B_0).
\]

\begin{figure}[h]
	  \centering
	\resizebox{.25\textwidth}{!}{
		
\begin{tikzpicture}[
	vertex/.style={circle, draw, fill=blue!10, thick, minimum size=13mm, font=\small},
	edge/.style={thick, draw=black},
	label_edge/.style={font=\scriptsize, fill=white, inner sep=1pt, sloped}
	]
	
	\foreach \i/\angle/\poids in {1/90/0.7, 2/0/0.3, 4/270/1.2, 5/180/1.3} {
		\node[vertex] (v\i) at (\angle:2cm) {(\i, \poids)};
	}
	
	\draw[edge] (v1) -- node[label_edge] {0.06} (v2);
	\draw[edge] (v2) -- node[label_edge] {0.14} (v4);
	\draw[edge] (v4) -- node[label_edge] {-0.22} (v5);
	\draw[edge] (v5) -- node[label_edge] {-0.18} (v1);
	
	\draw[edge] (v1) -- node[label_edge, pos=0.25, rotate=90] {0.22} (v4);
	\draw[edge] (v2) -- node[label_edge, pos=0.25] {-0.19} (v5);
	
\end{tikzpicture}
 
	}
	\caption{\label{initstate}Illustration of the initial condition represented as a graph: only one interaction per pair of vertices is shown for clarity, although each pair admits two directed interactions. Each vertex is of the form $(p, x_p)$, where $x_p$ denotes the abundance of species $p$. The edges are weighted by $w_{p,q}$, representing the ecological interaction strength between species $p$ and $q$, for every pair $(p,q)$.
	}
\end{figure}
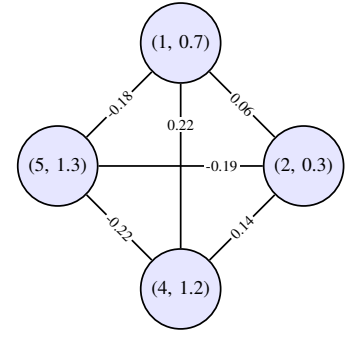

 This initial condition is illustrated as a graph in Fig.~\ref{initstate}, where, for clarity, only one relation is drawn per pair of vertices. We also consider the input \(U_{C_k} = (\mathbf{v}_{C_k}, \mathbf{\omega}_{L_k}, \mathbf{e}_{L_k})\) associated with the jump map \(j\) defined in \eqref{Defjump}, where \(L_k = C_k \otimes C_k\). For all \((t,k) \in \mathcal{T}\),

\begin{equation}
	\label{bacthera}
	C_k =
	\begin{cases}
		\{\mathbf{b}_1,\mathbf{b}_2,\mathbf{b}_4,\mathbf{b}_5, \mathbf{b}_9\}, & \text{if } t = 190, \\
		\{\mathbf{b}_1,\mathbf{b}_2,\mathbf{b}_3,\mathbf{b}_4,\mathbf{b}_5, \mathbf{b}_8, \mathbf{b}_9\}, & \text{if } t = 330, \\
		\{\mathbf{b}_1, \mathbf{b}_2, \mathbf{b}_4\}, & \text{if } t = 560, \\
		\varnothing, & \text{otherwise.}
	\end{cases}
\end{equation}

The input, modeling bacteriotherapy, is defined on the basis $C_k$ at different time points, as specified in Equation~\ref{bacthera}. It is illustrated as a graph in Fig.~\ref{inst4}. At times $t = 190$ days and $t = 330$ days, the input can be likened to Fecal Microbiota Transplantation (FMT). In the first case, it consists of adding species 9 along with its ecological interactions with the species present in the initial state (Fig.~\ref{initstate}). In the second case, it corresponds to adding species 3 and 8 to the state already enriched with species 9. At time $t = 560$ days, the process instead consists of removing certain species while specifying those that must be retained. Zero values on the vertices and edges of the input graphs can be interpreted as indicating no perturbation of the existing quantities.

We investigate several implementations\footnote{The C++ code implementing the proposed framework, including the model used in this study, is available upon request from the authors at \textit{christophe.roman@lis-lab.fr}.} of the hybrid model \(\mathcal{H}\) defined in \eqref{GenEquaDiff}, whose continuous dynamics are governed by \eqref{flowModel}.  The results are shown in Fig.~\ref{fig:microbiotaDyn}, Fig.~\ref{fig:microbiotaDynBacterio}, and Fig.~\ref{fig:microbiotaDynBacteriowithWeightsDyn}, which respectively illustrate: (i) the evolution of species abundances driven solely by the gLV dynamics, (ii) the evolution of species abundances with topological changes (species addition and removal), and (iii) the coupled evolution of species abundances, ecological interaction weights, and network topology. It should be noted, in general, that a comparison of these results with those of Jones et al.~\cite{jones2020navigation} for comparable scenarios (antibiotic effects and modification of species composition, without evolution of interaction weights) reveals a consistency in qualitative behaviors, despite the differences between the formalisms used.

\begin{figure}[h]
	\centering
	\begin{subfigure}[b]{0.55\columnwidth}
		\centering
		\resizebox{\linewidth}{!}{\resizebox{.7\textwidth}{!}{
	
	\begin{tikzpicture}[
		vertex/.style={circle, draw, fill=blue!10, thick, minimum size=13mm, font=\small},
		edge/.style={thick, draw=black},
		label_edge/.style={font=\small, fill=white, inner sep=1pt}
		]
		
		\foreach \i/\angle/\poids in {1/90/0.0, 2/0/0.0, 4/270/0.0, 5/180/0.0} {
			\node[vertex] (v\i) at (\angle:2.7cm) {(\i, \poids)};
		}
		
		\node[vertex, fill=red!10] (v9) at (0,0) {(9, 0.85)};
		
		\draw[edge] (v9) -- node[label_edge] {0.35} (v1);
		\draw[edge] (v9) -- node[label_edge] {-0.03} (v2);
		\draw[edge] (v9) -- node[label_edge] {0.67} (v4);
		\draw[edge] (v9) -- node[label_edge] {0.16} (v5);
		
	\end{tikzpicture}
	
}}
		\caption{\label{inst1}}
	\end{subfigure}%
	\\
	\vspace{2 mm}
	\begin{subfigure}[b]{0.75\columnwidth}
		\centering
		\resizebox{\linewidth}{!}{\resizebox{.5\textwidth}{!}{
	
\begin{tikzpicture}[
	vertex/.style={circle, draw, fill=blue!10, thick, minimum size=16mm, font=\normalsize},
	center_vertex/.style={circle, draw, fill=red!10, thick, minimum size=16mm, font=\small},
	edge/.style={thick, draw=black},
	label_edge/.style={font=\large\bfseries, fill=white, inner sep=2pt}
	]
	
	\node[vertex] (v1) at (140:6cm) {$(1, 0.0)$};
	\node[vertex] (v2) at (110:6cm) {$(2, 0.0)$};
	\node[vertex] (v4) at (80:6cm) {$(4, 0.0)$};
	\node[vertex] (v5) at (50:6cm) {$(5, 0.0)$};
	\node[vertex] (v9) at (20:6cm) {$(9, 0.0)$};
	
	\node[center_vertex] (v3) at (-3, -2) {$(3, 0.6)$};
	\node[center_vertex] (v8) at (3, -2) {$(8, 0.8)$};
	
	\draw[edge] (v3) -- node[label_edge] {$-0.77$} (v8);
	
	\draw[edge] (v3) -- node[label_edge, pos=0.5] {$0.14$} (v1);
	\draw[edge, bend left=10] (v3) -- node[label_edge, pos=0.3] {$-0.04$} (v2);
	\draw[edge, bend left=15] (v3) -- node[label_edge, pos=0.5] {$-0.13$} (v4);
	\draw[edge, bend left=20] (v3) -- node[label_edge, pos=0.25] {$-0.17$} (v5);
	\draw[edge, bend left=30] (v3) -- node[label_edge, pos=0.25] {$0.30$} (v9);
	
	\draw[edge, bend right=30] (v8) -- node[label_edge, pos=0.88] {$-0.4$} (v1);
	\draw[edge, bend right=20] (v8) -- node[label_edge, pos=0.85] {$-0.41$} (v2);
	\draw[edge, bend right=15] (v8) -- node[label_edge, pos=0.47] {$-1.01$} (v4);
	\draw[edge, bend right=10] (v8) -- node[label_edge, pos=0.7] {$0.55$} (v5);
	\draw[edge] (v8) -- node[label_edge, pos=0.5] {$0.44$} (v9);
	
\end{tikzpicture}
	
}}
		\caption{\label{inst2}}
	\end{subfigure}
	\\
	\vspace{2 mm}
	\begin{subfigure}[b]{0.5\columnwidth}
		\centering
		\resizebox{\linewidth}{!}{\resizebox{.13\textwidth}{!}{

\begin{tikzpicture}[
	vertex/.style={circle, draw, fill=blue!10, thick, minimum size=10mm},
	edge/.style={thick, draw=black},
	label/.style={font=\scriptsize, fill=white, inner sep=1pt}
	]
	
	\node[vertex] (v1) at (90:2cm) {$(1,0.0)$};
	\node[vertex] (v2) at (210:2cm) {$(2,0.0)$};
	\node[vertex] (v4) at (330:2cm) {$(4,0.0)$};
	
	\draw[edge] (v1) -- node[label] {$0.0$} (v2);
	\draw[edge] (v2) -- node[label] {$0.0$} (v4);
	\draw[edge] (v4) -- node[label] {$0.0$} (v1);
	
\end{tikzpicture}
}}
		\caption{\label{inst3}}
	\end{subfigure}
	\caption{\label{inst4}Illustration of the inputs of bacteriotherapy on the set of basis vectors $C_k$ (see equation \eqref{bacthera}) represented as a graph, with only one relation per pair of vertices shown for clarity: (a) addition of species 9 at $t = 190$ days; (b) addition of species 3 and 8 at $t = 330$ days; (c) removal of all species except 1, 2, and 4 at $t = 560$ days.}
\end{figure}
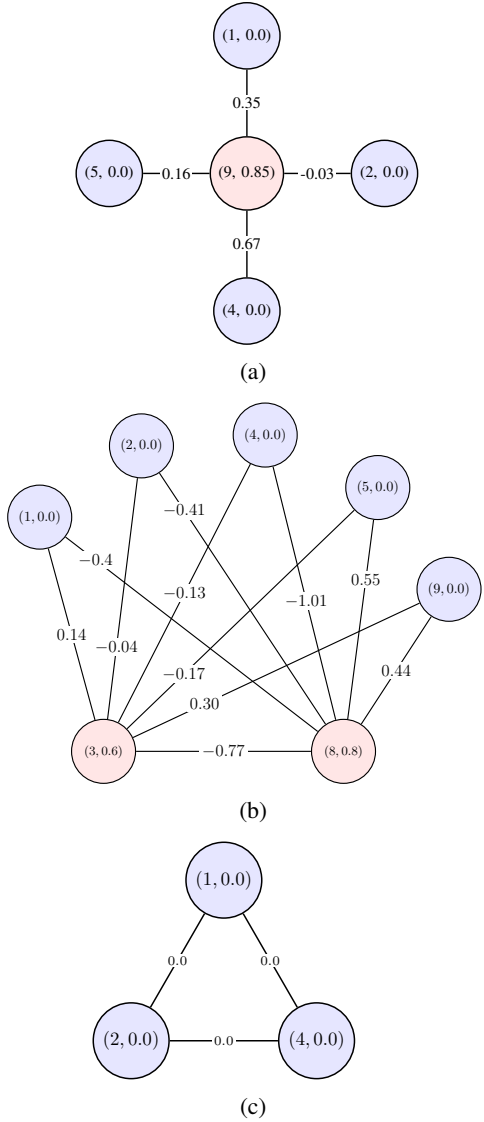

The purpose of Fig.~\ref{fig:microbiotaDyn} is to illustrate the effect of antibiotic administration on the evolution of species abundances (the attributes), by comparing the case without antibiotics (Fig.~\ref{fig:sub1}) to the case with a 3-day antibiotic treatment (Fig.~\ref{fig:sub2}).  This administration slows the process by which the system reaches equilibrium. Fig.~\ref{fig:microbiotaDynBacterio} illustrates the bacteriotherapy, which consists of adding or removing certain species carried by the base vectors of the set $C_k$ defined in \eqref{bacthera}.  This intervention results in a structural change over time. The comparison between subfigures \ref{fig:sub3} and \ref{fig:sub4} highlights the importance of distinguishing a true species absence from an abundance that is merely close to zero. This distinction, carefully discussed in \cite{doliveira2026generalized}, depends on the underlying formalism. In the framework of Jones et al.~\cite{jones2020navigation}, species are represented by a fixed-size vector, where the absence of a species corresponds to an abundance exactly equal to zero. In contrast, in the framework proposed in this paper, where the representation has variable size, species absence is encoded by the absence of the corresponding basis vector itself. This difference helps mitigate the accumulation of numerical errors, since infinitesimal noise around a zero value can, in fixed-size models, spuriously trigger the evolution of a species that should be absent. This behavior is precisely what is observed when comparing subfigures \ref{fig:sub3} and \ref{fig:sub4}, depending on whether $J_-$ is deactivated or activated.

Fig.~\ref{fig:microbiotaDynBacteriowithWeightsDyn} illustrates the complete dynamics of the system state, including its attributes, weights, and topology.  Topology changes occur through bacteriotherapy, carried by the basis vectors of the set $C_k$ (see \eqref{bacthera}).  The addition of weight dynamics, based on a model adapted from Oja's rule \cite{dayan2005theoretical}, rapidly balances attribute values around a common level, even when there is bacteriotherapy, which contrasts with the observation in Fig.~\ref{fig:microbiotaDynBacterio}.

\begin{figure}[t]
	\centering
	
	\begin{subfigure}[b]{0.5\textwidth}
		\centering
		\includegraphics[width=0.99\textwidth]{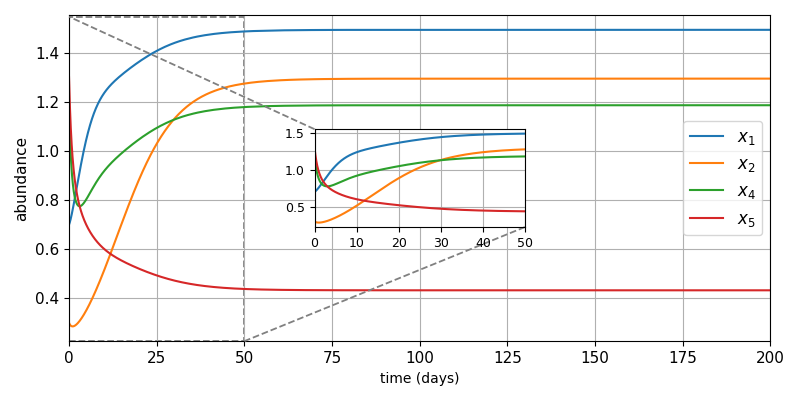}
		\caption{}
		\label{fig:sub1}
	\end{subfigure}
	
	
	\begin{subfigure}[b]{0.5\textwidth}
		\centering
		\includegraphics[width=0.99\textwidth]{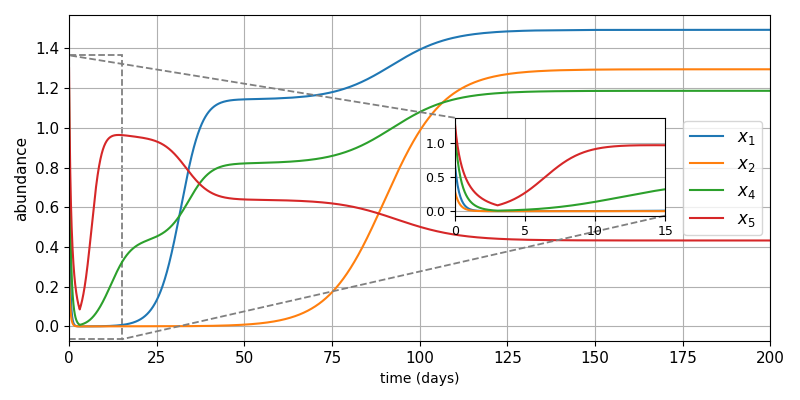}
		\caption{}
		\label{fig:sub2}
	\end{subfigure}
	
	
	\caption{Microbiota evolution simulated by the continuous model~\eqref{flowModel}, without weight dynamics and with a fixed topology: (a) no antibiotic treatment; (b) antibiotic treatment during the first three days.}
	\label{fig:microbiotaDyn}
\end{figure}

\begin{figure}[t]
	\centering
	
	\begin{subfigure}[b]{0.5\textwidth}
		\centering
		\includegraphics[width=0.85\textwidth]{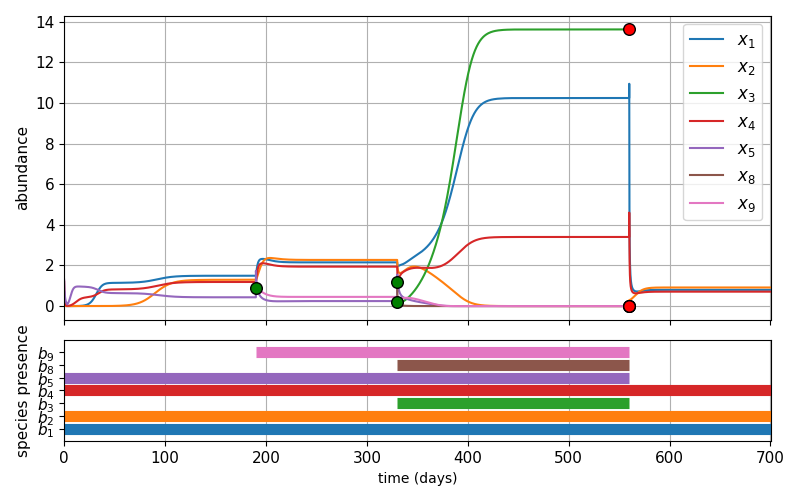}
		\caption{}
		\label{fig:sub3}
	\end{subfigure}
	
	
	\begin{subfigure}[b]{0.5\textwidth}
		\centering
		\includegraphics[width=0.85\textwidth]{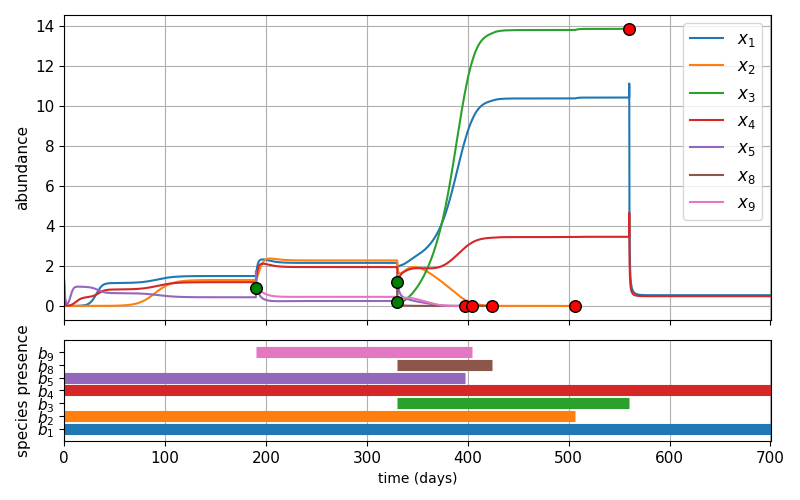}
		\caption{}
		\label{fig:sub4}
	\end{subfigure}
	
	
	\caption{Microbiota evolution under the hybrid model $\mathcal{H}$ (Eqs. \eqref{GenEquaDiff} and \eqref{flowModel}) during 3 days of antibiotics, followed by bacteriotherapy $U_{C_k}$ (Eq. \eqref{bacthera}), without weight dynamics and without $J_+$ activation (Eq. \eqref{jumpset}): (a) $J_-$ off, (b) $J_-$ on.}
	\label{fig:microbiotaDynBacterio}
\end{figure}

\begin{figure}[t]
	\centering
	
	\begin{subfigure}[b]{0.5\textwidth}
		\centering
		\includegraphics[width=0.85\textwidth]{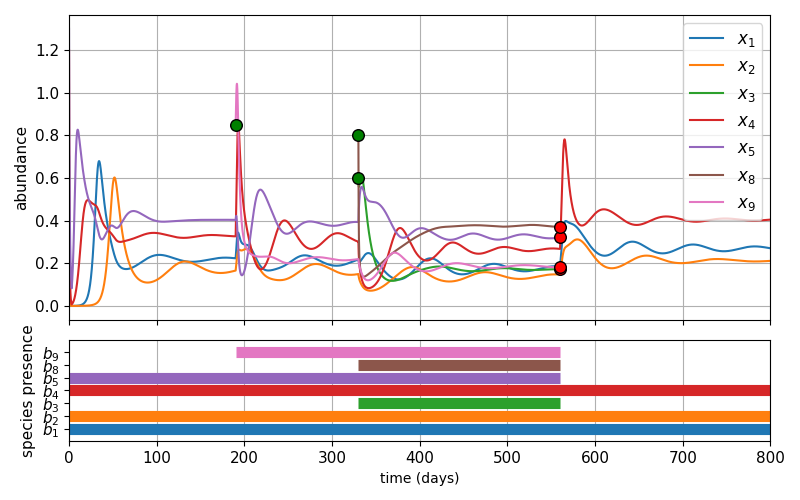}
		\caption{}
		\label{fig:sub5}
	\end{subfigure}
	
	
	\begin{subfigure}[b]{0.5\textwidth}
		\centering
		\includegraphics[width=0.85\textwidth]{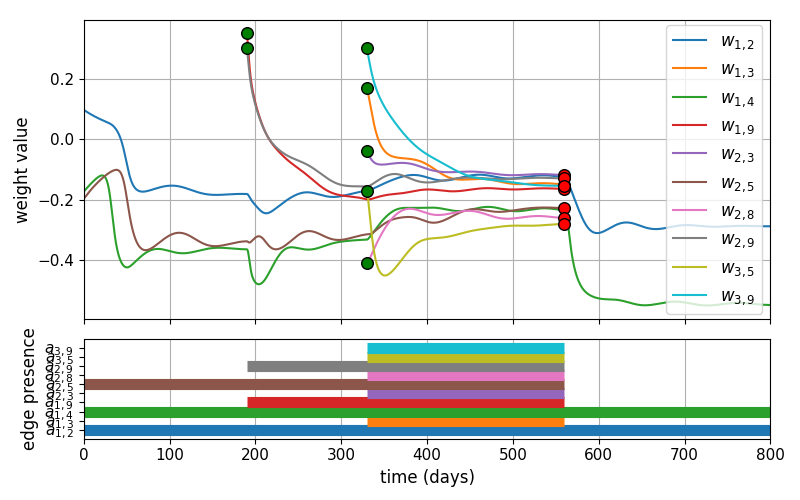}
		\caption{}
		\label{fig:sub6}
	\end{subfigure}
	
	
	\caption{Microbiota evolution under the hybrid model $\mathcal{H}$ (Eqs.~\eqref{GenEquaDiff} and~\eqref{flowModel}) during three days of antibiotic treatment, followed by bacteriotherapy $U_{C_k}$ (Eq.~\eqref{bacthera}): (a) attribute evolution; (b) weight evolution, with the evolution of selected edges and weights also displayed.}
	\label{fig:microbiotaDynBacteriowithWeightsDyn}
\end{figure}

\section{\label{Conc}Conclusion}
\noindent In this paper, we have developed a framework guaranteeing the existence of trajectories for a dynamical system evolving in the graph state-space, which exhibits both discrete and continuous dynamics. We introduced an algebraic structure to model discrete changes in the topology, node attributes, and edge weights, while continuous evolution is governed by differential equations acting on attributes and weights. This construction reveals, to the best of our knowledge, a new mathematical structure: semi-vector spaces over fields, extending the classical notion usually defined over semi-fields. The framework naturally fits within hybrid dynamical systems theory and, in a variable-basis setting, provides a rigorous foundation for proving existence of solutions. Its relevance is illustrated through simulations of gut microbiota dynamics, including bacteriotherapy and fecal microbiota transplantation (FMT), interpreted as control inputs. The results reproduce qualitative trends observed in the literature. Future work will focus on stability and controllability analysis, as well as stochastic extensions within both the graph state space and the variable-basis framework.

\appendix
\subsection{Proofs of the Propositions of Theorem~\ref{GraphSetStructureU+}}
\label{ProofPropTheoU+}

\begin{proof}(Proposition \ref{PropAddUnion})
The additive union inherits the commutative and associative properties of union and addition. Regarding the property of additive cancellation, a counterexample is sufficient to show that it does not apply. By taking \(X\) and \(Z\) as two distinct zero subgraphs, \(X^0\) and \(Z^0\), of \(Y\), we have \(X^0 \cup_+ Y = Z^0 \cup_+ Y \Rightarrow Y = Y\), while \(X^0 \neq Z^0\).
\end{proof}

\begin{proof}(Proposition \ref{AddUnionIdentity})
For any graph \( G \in \mathbb{G} \), a graph \( X \in \mathbb{G} \) is the identity in the structure \( (\mathbb{G}, \cup_+) \) if and only if it satisfies the following condition:

\begin{equation}
\label{DefElementNeutre}
G\cup_+ X = X \cup_+ G = G.
\end{equation}

The empty graph \(\emptyset_{\mathbb{G}}\) fully satisfies the condition in \eqref{DefElementNeutre}. Now, suppose that \(X \neq \emptyset_{\mathbb{G}}\). We will demonstrate that only the empty graph \(\emptyset_{\mathbb{G}}\) satisfies this condition. Let \(X = (V_X, E_X)\) and \(G = (V_G, E_G)\).

\begin{itemize}
  \item[--]  First, equation \eqref{DefElementNeutre} indicates that the cardinalities of \( V_G \) and \( E_G \) remain unchanged when applying \( X \). Specifically, this gives:  $\lvert \mathcal{I}_G \rvert + \lvert \mathcal{I}_X \rvert - \lvert \mathcal{I}_G \cap \mathcal{I}_X \rvert = \lvert \mathcal{I}_G \rvert$, which leads to \(\lvert \mathcal{I}_X \rvert = \lvert \mathcal{I}_G \cap \mathcal{I}_X \rvert\). Therefore, \( \mathcal{I}_X \subseteq \mathcal{I}_G \), and similarly, \( \mathcal{L}_X \subseteq \mathcal{L}_G\). This means that the vertices in \( V_X \) have the same labels as some of the vertices in \( V_G \), and the edges in \( E_X \) have the same pairs of labels as some of the edges in \( E_G \). However, this result does not provide any information about the values of attributes or weights.
  
  \item[--] Secondly, assuming that \( V_G = \{ (k, g_k) \}_{k \in \mathcal{I}_G} \) and \( V_X = \{ (k, x_k) \}_{k \in \mathcal{I}_X} \), the condition \eqref{DefElementNeutre} along with the result that \( \mathcal{I}_X \subseteq \mathcal{I}_G \) leads to:
   \begin{equation}
 \label{VVp}
    V_G=\left\{ (k,g_k+x_k) \mid k\in \mathcal{I}_X \right\} \cup \left\{ (k,g_k)\mid k\in \mathcal{I}_G\setminus \mathcal{I}_X \right\}.
 \end{equation}
For equation \eqref{VVp} to be valid, the first member must match the second member. This requires that for all \( k \in \mathcal{I}_X \), the condition \( g_k = g_k + x_k \) holds, leading to \( x_k = 0 \). Similarly, if we consider  \(E_G = \left\{ ((k,l), w_{p,q}) \mid (k,l) \in \mathcal{L}_G \right\}\)  and \(E_X = \left\{ ((k,l), v_{k,l}) \mid (k,l) \in \mathcal{L}_X \right\} \), it follows that \( v_{k,l} = 0 \). Consequently, the null subgraphs of \( G \) satisfy condition \eqref{DefElementNeutre}.

\item[--]In conclusion, based on the above, only the empty graph \( \emptyset_{\mathbb{G}} \) and the null sub-graphs of \( G \) satisfy condition \eqref{DefElementNeutre}. Let \( X^0\) represent a null sub-graph of \( G \). By applying the definition of condition 3 to the empty graph \( \emptyset_{\mathbb{G}} \) and the null graph \( X^0 \), the following results are obtained: $\emptyset_{\mathbb{G}}\cup_+ X^0= X^0 \cup_+ \emptyset_{\mathbb{G}} = X^0$. Thus, \(\emptyset_{\mathbb{G}}\) is the only identity of \((\mathbb{G}, \cup_+)\).
\end{itemize}
\end{proof}

\begin{proof}(Proposition \ref{Symetric+U})
A graph \( G \in \mathbb{G} \) is considered invertible if and only if there exists a graph \( X \in \mathbb{G} \) such that:

\begin{equation}
\label{UnionInverse}
G \cup_+ X = X \cup_+ G = \emptyset_{\mathbb{G}}.
\end{equation}
Let's assume that $G=(V_G,E_G)$  and $X = (V_X,E_X)$ with. For the set of vertices \(V_G = \{(i, g_i)\}_{i \in \mathcal{I}_G}\) and \(V_X = \{(i, x_i)\}_{i \in \mathcal{I}_X}\), relation \eqref{UnionInverse} would induce: 
\[\lvert \mathcal{I}_G \rvert + \lvert \mathcal{I}_X \rvert - \lvert \mathcal{I}_G \cap \mathcal{I}_X \rvert = 0,\]
 which implies:

\begin{equation}
\label{InverseProof}
  \lvert \mathcal{I}_G \rvert + \lvert \mathcal{I}_X \rvert = \lvert \mathcal{I}_G \cap \mathcal{I}_X \rvert.
\end{equation}

Since \( \mathcal{I}_G \cap \mathcal{I}_X \subseteq \mathcal{I}_G \) and \( \mathcal{I}_G \cap \mathcal{I}_X \subseteq \mathcal{I}_X \) are always true, we have:  $ \lvert \mathcal{I}_G \rvert \geq \lvert \mathcal{I}_G \cap \mathcal{I}_X \rvert$ and  $\lvert \mathcal{I}_X \rvert \geq \lvert \mathcal{I}_G \cap \mathcal{I}_X \rvert$, which leads to:

 \begin{equation}
   \label{Cas3Inverse}
    \lvert \mathcal{I}_G \rvert + \lvert \mathcal{I}_X \rvert \geq 2 \lvert \mathcal{I}_G \cap \mathcal{I}_X \rvert.
 \end{equation}
  The only situation in which both relations \eqref{InverseProof} and \eqref{Cas3Inverse} hold true is \( \mathcal{I}_G = \mathcal{I}_X = \emptyset \). By similar reasoning for the edge sets  \(E_G = \left\{ (k,l, w_{k,l}) \mid (k,l) \in \mathcal{L}_G \right\}\) and \(E_X = \left\{ (k,l, v_{k,l}) \mid (k,l) \in \mathcal{L}_X \right\}\), we conclude that \( \mathcal{L}_G = \mathcal{L}_X = \emptyset \). Since the graph \( G \) is indexed by \(\mathcal{I}_G\) and \(\mathcal{L}_G\), and \( X \) is indexed by \(\mathcal{I}_X\) and \(\mathcal{L}_X\), the only solution that satisfies \eqref{UnionInverse} is \( G = X = \emptyset_{\mathbb{G}} \).
\end{proof}
\subsection{Proofs of the Propositions of Theorem \ref{GraphSetStructuren+}}
\label{ProoFTheo2}

\begin{proof}(Proposition \ref{PropAddInter})
The additive intersection $\cap_+$ inherits the commutativity and associativity properties from both intersection and addition. The non-additive cancellation is verified simply by considering $X$ and $Z$ as two distinct graphs in $\mathbb{G}$ that contain the null graph $Y^0$.
\end{proof}

\begin{proof}(Proposition \ref{IdAddInter})
  Let \( G = (V_G, E_G) \in \mathbb{G} \) and \( X = (V_X, E_X) \in \mathbb{G} \). The graph \( X \) is the identity element in the structure \((\mathbb{G}, \cap_+)\) if and only if

  \begin{equation}
    \label{NeutreDefInterOver}
    G \cap_+ X = X \cap_+ G = G.
  \end{equation}

  Assume that \( V_G = \{(i, g_i)\}_{i \in \mathcal{I}_G} \) and \( V_X = \{(i, x_i)\}_{i \in \mathcal{I}_X} \). From \eqref{NeutreDefInterOver}, we deduce that  $ \lvert \mathcal{I}_G \rvert = \lvert \mathcal{I}_G \cap \mathcal{I}_X \rvert,$  which implies that $\mathcal{I}_G \subseteq \mathcal{I}_X$. Applying Definition~\ref{AddInterDef} and using \eqref{NeutreDefInterOver} along with \( \mathcal{I}_G \subseteq \mathcal{I}_X \), we obtain  
   \[
   \{(i, g_i + x_i) \mid i \in \mathcal{I}_G\} = \{(i, g_i) \mid i \in \mathcal{I}_G\},
   \]
   which leads to  $ x_i = 0 \quad \text{for all } i \in \mathcal{I}_G$. Similarly, for the edges, let $E_G = \{(k, l, w_{k,l}) \mid (k, l) \in \mathcal{L}_G\}$ and  $E_X = \{(k, l, v_{k,l}) \mid (k, l) \in \mathcal{L}_X$. From \eqref{NeutreDefInterOver}, it follows that $ \mathcal{L}_G \subseteq \mathcal{L}_X $ and $v_{k,l} = 0 \quad \text{for all } (k,l) \in \mathcal{L}_G$. Thus, the graph \( X \in \mathbb{G} \) that satisfies the constraints \( \mathcal{I}_G \subseteq \mathcal{I}_X \), \( \mathcal{L}_G \subseteq \mathcal{L}_X \), \( x_i = 0 \), and \( v_{k,l} = 0 \) for every graph \( G \in \mathbb{G} \) is none other than the total, complete, and zero graph \( \mathcal{G}_{t,c}^0 \).
\end{proof}

\begin{proof}(Proposition \ref{SymInter+})
A graph \( G \in \mathbb{G} \) is considered invertible if and only if there exists a graph \( X \in \mathbb{G} \) such that:

\begin{equation}
\label{InterInverse}
G \cap_+ X = X \cap_+ G = \mathcal{G}_{t,c}^0.
\end{equation}

Suppose \( G = (V_G, E_G) \) and \( X = (V_X, E_X) \) with vertex sets \( V_G = \{(i, g_i)\}_{i \in \mathcal{I}_G} \) and \( V_X = \{(i, x_i)\}_{i \in \mathcal{I}_X} \). The relation \eqref{InterInverse} implies: $\lvert \mathcal{I}_G \cap \mathcal{I}_X \rvert = \lvert \mathcal{I} \rvert$, where \( \mathcal{I} \) is the set of labels of \( \mathcal{G}_{t,c}^0 \). This result leads to two conclusions: (1) \( \mathcal{I}_G \geq \mathcal{I} \) and \( \mathcal{I}_X \geq \mathcal{I} \). Since \( \mathcal{I} \) is the maximal set, we conclude that \( \mathcal{I}_G = \mathcal{I}_X = \mathcal{I} \); and (2) From Definition \ref{AddInterDef} and the relation \eqref{InterInverse}, we have \( g_i = -x_i \). Thus, only total and complete graphs \( \mathcal{G}_{t,c} \) are invertible under the operation \( \cap_+ \).
\end{proof}

\section*{References}
\bibliography{tacref}
\bibliographystyle{IEEEtran}

\begin{IEEEbiography}[{\includegraphics[width=1in,height=1.25in,clip,keepaspectratio]{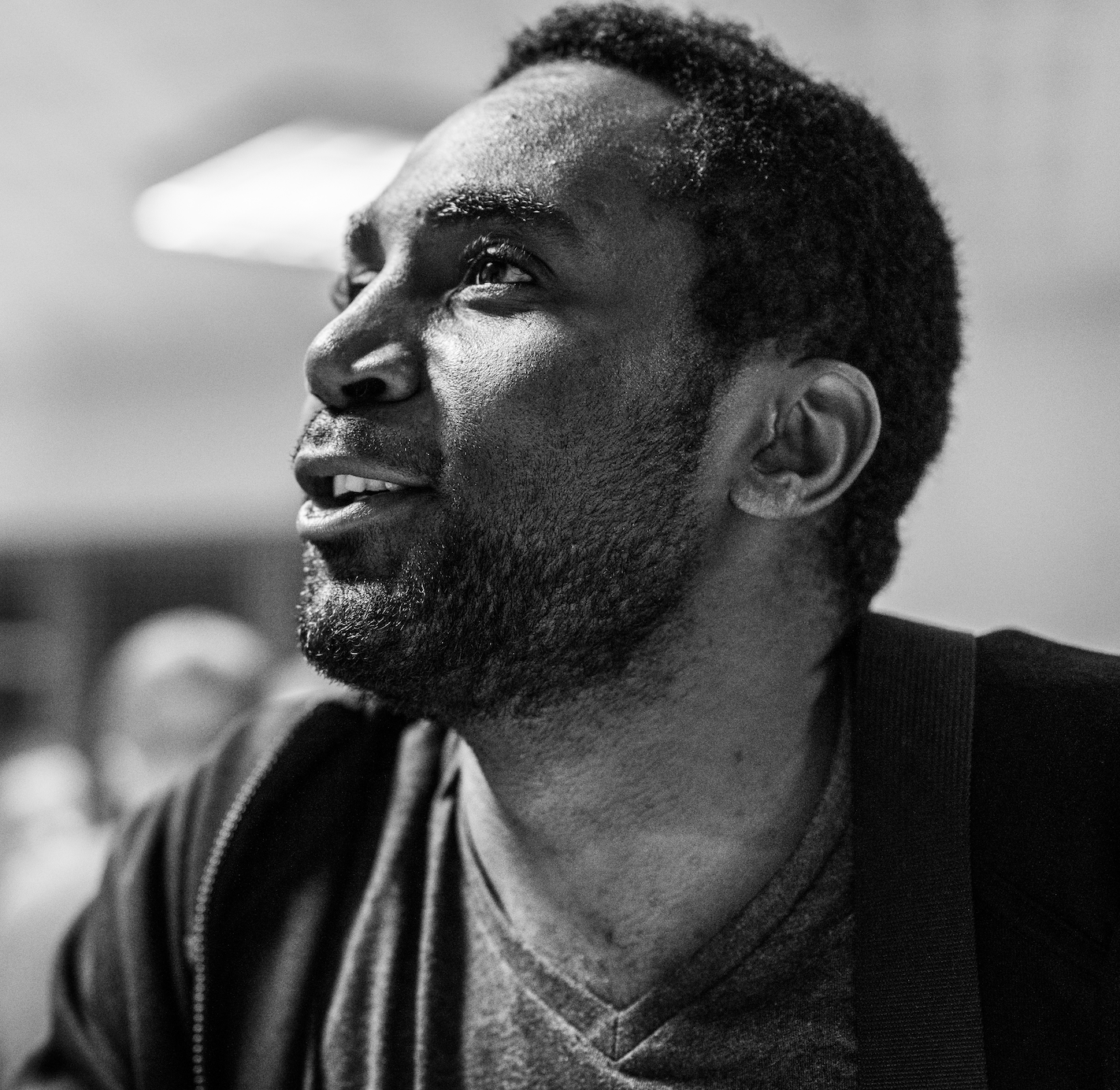}}]{Arthur Doliveira} 
	received the Master’s degree in Applied Mathematics from the University of Rennes, France, in 2020. He is currently pursuing a Ph.D. degree in Automatic Control at Aix-Marseille University, France, within the Laboratoire d’Informatique et des Systèmes (LIS), from 2022 to 2026. His research interests include modelling collective dynamics involving structural changes in complex systems.
\end{IEEEbiography}

\begin{IEEEbiography}[{\includegraphics[width=1in,height=1.25in,clip,keepaspectratio]{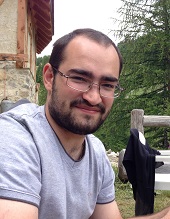}}]{Christophe Roman} 
	was born in Val-des-pr\`es, France, in 1991. He graduated in applied physics and electrical engineering from the Ecole Normale Sup\'erieure de Cachan, France in 2015. He received the Ph. D degree in Automatic from the University Grenoble Alpes in 2018. Since 2020, he has been an associate professor at Aix-Marseille University and a member of the Laboratory Infomatique and System (LIS). His main research interests are data-based diagnostic and EDP parameter identification and control. 
\end{IEEEbiography}

\begin{IEEEbiography}[{\includegraphics[width=1in,height=1.25in,clip,keepaspectratio]{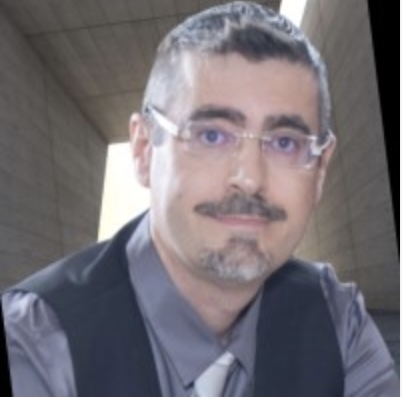}}]{Guillaume Graton} 
received the Ph. D degree in Automatic control from the University of Orléans in 2005. Since 2007, he is an associate professor at Centrale Méditerranée and a member of the LIS (Laboratoire d'Infomatique et des Systèmes). His research interests are diagnosis, prognosis, system modeling. He has many industrial collaborations and more than 60 papers published in international conferences and journals.
\end{IEEEbiography}

\begin{IEEEbiography}[{\includegraphics[width=1in,height=1.25in,clip,keepaspectratio]{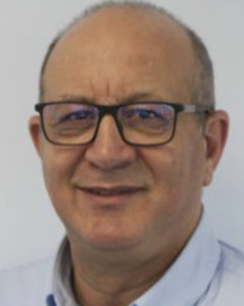}}]{Mustapha Ouladsine}  received the Ph.D. degree in automation from the University Henri Poincarré, Nancy, France, in 1993. He is currently a Full Professor with the University of Aix-Marseille (AMU), France. He is a Professor of data science, neural networks, and diagnostics with AMU. He has served as an Editor of some IEEE conferences. He is the author of several international publications and books on fault detection and isolation in dynamic systems. He has supervised more than 26 Ph.D.s. Since 2016, he has been interested in the use of artificial intelligence and, more specifically, in competitive learning by neural networks for diagnosis in health. His research focuses on the health monitoring of dynamic systems.
\end{IEEEbiography}

\end{document}